\documentclass[a4paper]{amsart}

\usepackage[english]{babel}

\usepackage{amssymb}
\usepackage{mathrsfs}
\usepackage{hyperref}
\usepackage{xcolor}
\usepackage{eucal}

\usepackage{amssymb}
\usepackage{amsmath}
\usepackage{amsthm}
\usepackage{imakeidx}
\usepackage{breqn}

\usepackage{diagbox}

\usepackage{verbatim}

\usepackage{fancyhdr}

\usepackage{tikz-cd}

\usepackage[utf8]{inputenc}

\usepackage{graphicx}

\usepackage{xcolor}

\usepackage[style=alphabetic, url=false, doi=false, backend=bibtex]{biblatex}

\theoremstyle{definition}

\newtheorem{mydef}{Definition}[section]

\newtheorem*{myass}{Assumption}

\newtheorem*{myack}{Acknowledgments}

\theoremstyle{remark}

\newtheorem{mybem}[mydef]{Remark}

\theoremstyle{plain}

\newtheorem{mysen}[mydef]{Theorem}
\newtheorem{mylem}[mydef]{Lemma}

\newtheorem{myfact}[mydef]{Fact}
\newtheorem{myclaim}{Claim}
\newtheorem*{mysenx}{Theorem}

\numberwithin{mydef}{section}

\DeclareMathOperator{\cof}{cof}
\DeclareMathOperator{\dom}{dom}

\DeclareMathOperator{\im}{im}

\DeclareMathOperator{\otp}{otp}
\DeclareMathOperator{\cf}{cf}

\DeclareMathOperator{\Coll}{Coll}

\DeclareMathOperator{\stem}{stem}
\DeclareMathOperator{\osucc}{osucc}
\DeclareMathOperator{\On}{On}

\newcommand{\dA}{\mathbb{A}}

\newcommand{\dI}{\mathbb{I}}

\newcommand{\dL}{\mathbb{L}}

\newcommand{\dP}{\mathbb{P}}
\newcommand{\dQ}{\mathbb{Q}}
\newcommand{\dR}{\mathbb{R}}

\newcommand{\dT}{\mathbb{T}}

\newcommand{\uhr}{\upharpoonright}

\newcommand{\ZFC}{\mathsf{ZFC}}
\newcommand{\GCH}{\mathsf{GCH}}

\newcommand{\AP}{\mathsf{AP}}
\newcommand{\LIP}{\mathsf{LIP}}
\newcommand{\CH}{\mathsf{CH}}

\title[Failure of Approachability]{Total Failure of Approachability at Successors of Singulars of Countable Cofinality}

\author{Hannes Jakob}
\address{225 E Mulberry St, Denton, TX 76201}
\email{hannes.jakob@unt.edu}


\date{\today}

\begin{document}
	
	
	
	
	\begin{abstract}
		Relative to class many supercompact cardinals, we construct a model of $\ZFC+\GCH$ where for every singular cardinal $\delta$ of countable cofinality and every regular uncountable $\mu<\delta$ there are stationarily many non-approachable points of cofinality $\mu$ in $\delta^+$. This answers a question of Mitchell and provides a decisive answer to a question of Foreman and Shelah.
	\end{abstract}
		
	\maketitle
	
	The notion of independence is fundamental to set-theoretic research. It refers to the phenomenon that many set-theoretic statements are neither proven nor refuted by $\ZFC$, the most commonly accepted axioms of set theory. One particular theme in the study of independence statements is the fact that successors of singular cardinals are often much less malleable than successors of regular cardinals -- recall that a cardinal $\kappa$ is \emph{singular} if we can write $\kappa=\bigcup_{i\in x}y_i$, where $x$ and every $y_i$ has size ${<}\,\kappa$ and \emph{regular} otherwise. Additionally, many independence results for regular cardinals cannot be obtained for singulars. The most famous example of this phenomenon is the combination of two results due to William Easton and Jack Silver: Easton showed that the continuum function -- mapping a cardinal $\kappa$ to the size of its powerset -- when restricted to the class of \emph{regular} cardinals can be equal to any function, provided that function is nondecreasing and maps every cardinal to a cardinal with a larger cofinality. On the other hand, Silver showed that whenever $\kappa$ is a singular cardinal with uncountable cofinality -- meaning that any function from $\omega$ into $\kappa$ is bounded -- and $2^{\nu}=\nu^+$ for every cardinal $\nu<\kappa$, then $2^{\kappa}=\kappa^+$. Thus, Easton's result cannot be replicated for the class of \emph{all} cardinals. And even at $\aleph_{\omega}$, the first singular cardinal, Saharon Shelah demonstrated provable constraints regarding the continuum function: He showed that, assuming $2^{\aleph_n}<\aleph_{\omega}$ for all $n<\omega$, $2^{\aleph_{\omega}}<\aleph_{\omega_4}$.
	
	Another example highlighting the difference between successors of singular cardinals and successors of regular cardinals is Shelah's \emph{approachability ideal}: Shelah introduced the ideal $I[\kappa^+]$ in \cite{ShelahSuccSingCard} in order to study the indestructibility of stationary subsets of $\kappa^+$ under sufficiently closed forcings and did so both for singular and regular cardinals $\kappa$. Despite the definition of $I[\kappa^+]$ being the same in both cases, its behavior differs greatly depending on whether $\kappa$ is regular or singular: When $\kappa$ is a regular cardinal, the set $E_{<\kappa}^{\kappa^+}$, consisting of all ordinals below $\kappa^+$ with cofinality ${<}\,\kappa$, is always a member of $I[\kappa^+]$ and so the behavior of $I[\kappa^+]$ is determined by which subsets of $E_{\kappa}^{\kappa^+}$ it contains. On the other hand, if $\kappa$ is singular, it was only proven that the set $E_{\leq\cf(\kappa)}^{\kappa^+}$ belongs to $I[\kappa^+]$ and so many more options regarding the behavior of $I[\kappa^+]$ seem possible.
	
	Recently, in joint work of Maxwell Levine and the author (see \cite{JakobLevineFailureApproachability}), we solved an almost four-decade-old open problem due to Saharon Shelah by constructing a model -- for an arbitrary $n\in(0,\omega)$ -- where $E_{\aleph_n}^{\aleph_{\omega+1}}\notin I[\aleph_{\omega+1}]$. Previously it had only been known that there could be stationarily many non-approachable points of cofinality $\aleph_1$ in $\aleph_{\omega+1}$. Our result begs the natural question (which was also raised by Mitchell in \cite{MitchellNonstationary} and by Dobrinen in private communication) whether it is consistent for a singular cardinal $\kappa$ of uncountable cofinality that $E_{\mu}^{\kappa^+}\notin I[\kappa^+]$ for \emph{every} regular cardinal $\mu\in(\aleph_0,\kappa)$.
	
	In this paper, we answer the preceding question and even answer it for every cardinal of countable cofinality simultaneously. Since $I[\kappa^+]$ contains a stationary subset of $E_{\mu}^{\kappa^+}$ for every regular $\mu<\kappa$ (see \cite[Theorem 9.2]{CummingsNotes}), the following theorem shows that the greatest possible failure of the approachability property at successors of singular cardinals of countable cofinality is actually consistent:
	
	\begin{mysenx}
		Assume that there is a class of supercompact cardinals. There exists a class forcing extension in which, for every singular cardinal $\delta$ with countable cofinality and every regular uncountable $\mu<\delta$, the set $E_{\mu}^{\delta^+}$ is not a member of $I[\delta^+]$.
	\end{mysenx}
	
	In particular, whenever $\delta$ is singular with countable cofinality, $\AP_{\delta}$, which states that $\delta^+\notin I[\delta^+]$, does not hold. Since $\AP_{\delta}$ is implied by the \emph{weak square principle} $\square_{\delta}^*$, the use of large cardinals is provably necessary for this result (see e.g. \cite{SargsyanSquareFailure}).
	
	Our main result is made possible by a technical advancement of the main crux of our previous paper (see \cite[Theorem 4.11]{JakobLevineFailureApproachability}). The major obstacle towards obtaining models in which a stationary set of a given cofinality is not in $I[\delta^+]$ is finding posets which can collapse successors of singular cardinals without making them \emph{$d$-approachable}. Whether this was possible without preserving the successor of the singular as a cardinal had been open for around four decades and was a major obstacle both in non-approachable sets concentrating on specific cofinalities and iterating such constructions to obtain the failure of $\AP$ for many cardinals simultaneously. However, we solved that issue in \cite{JakobLevineFailureApproachability} by constructing a poset which, under suitable hypotheses, collapses the successor of a singular cardinal of countable cofinality to any desired cofinality without making it $d$-approachable. That theorem relied on the tension between a weak form of the approximation property (see \cite[Theorem 4.3.]{JakobLevineFailureApproachability}) and our iteration not adding new short functions into small sets (see \cite[Lemma 4.10]{JakobLevineFailureApproachability}). In this work, we achieve a strengthening of \cite[Theorem 4.3]{JakobLevineFailureApproachability} (see Theorem \ref{BetterApproxProp}) which allows us to dispense with \cite[Lemma 4.10]{JakobLevineFailureApproachability}. Using this, we are able to replace the Namba forcing used there by a new forcing notion which, similarly to supercompact Prikry forcing, uses elements of $[\lambda]^{<\kappa}$ as opposed to ordinals. This forcing is constructed using higher variants of the $\LIP$-ideals which we used in \cite{JakobLevineFailureApproachability} and which have previously been considered by Matsubara (see \cite{MatsubaraLIP}). In technical terms, the poset from \cite{JakobLevineFailureApproachability} uses infinitely many generically measurable cardinals while our new poset only uses a single generically supercompact cardinal. Despite the latter being a stronger hypothesis in terms of its consistency strength, it is much easier to obtain in our case.
	
	Using this poset, we first construct a forcing which, given a regular cardinal $\mu$ and a supercompact cardinal $\kappa>\mu$, collapses $\kappa$ to become $\mu^+$ while forcing that for every singular cardinal $\delta>\kappa$ with countable cofinality, $E_{\mu}^{\delta^+}\notin I[\delta^+]$. We then define a class-length iteration of instances of that forcing which gives us the model of the main theorem.
	
	The paper is organized as follows: After giving preliminary definitions and results, we construct the tools we need to carry out the main construction: In Section 2, we prove iteration theorems and other statements regarding Prikry-type forcings. In Section 3, we investigate the Laver-Ideal Property which is used in the definition of the Namba forcing. In Section 4, we define a forcing which collapses the successor of a singular cardinal without making it approachable and derive its properties. In Section 5, we define an iteration of instances of the forcing from Section 4 and show that it forces, for a specific $\mu$, that for every singular $\delta>\mu$ with countable cofinality, there are stationarily many non-approachable points in $\delta^+$ with cofinality $\mu$. Finally, in Section 6, we construct an iteration of instances of the forcing from Section 5 and prove our main theorem.
	
	\begin{myack}
		The author wants to thank the referee for their diligent reading and valuable comments, leading to an improvement of the manuscript.
	\end{myack}
	
	\section{Preliminaries}
		
	We will assume the reader is familiar with the basics of forcing, large cardinals and elementary embeddings. Good introductory material can be found in the textbooks by Jech (see \cite{JechSetTheory}), Kunen (see \cite{KunenSetTheory}) and Kanamori (see \cite{KanamoriHigherInfinite}). More specialized material relevant to this paper appears in the chapters by Cummings (regarding elementary embeddings; see \cite{CummingsHandbook}), Foreman (regarding large ideals on successor cardinals; see \cite{ForemanChapter}), Eisworth (regarding the behavior of successors of singular cardinals; see \cite{EisworthHandbook}) and Gitik (regarding properties and iterations of Prikry-type forcings; see \cite{GitikHandbook}) in the Handbook of Set Theory.
	
	In this section, we give preliminary definitions and results which will be used throughout this paper.
	
	\subsection{Approachability, Colorings and Supercompact Cardinals}\hfill
	
	The approachability ideal was introduced by Shelah in \cite{ShelahSuccSingCard} in order to obtain results regarding the indestructibility of stationary sets under sufficiently closed forcings:
	
	\begin{mydef}
		Let $\mu$ be a cardinal. A set $S\subseteq\mu^+$ is in the \emph{approachability ideal} $I[\mu^+]$ if there is a club $C\subseteq\mu^+$ and a sequence $(x_i)_{i<\mu^+}$ of elements of $[\mu^+]^{<\mu}$ such that for any $\gamma\in S\cap C$, $\gamma$ is \emph{approachable with respect to $(x_i)_{i<\mu^+}$}, i.e. there is $A\subseteq\gamma$ unbounded of minimal order-type such that whenever $\beta\in\gamma$, $A\cap\beta\in\{x_i\;|\;i<\gamma\}$.
		
		The \emph{approachability property $\AP_{\mu}$} states that $I[\mu^+]$ is improper, i.e. $\mu^+\in I[\mu^+]$.
	\end{mydef}
	
	Although not directly relevant to the contents of this manuscript, the approachability ideal $I[\mu^+]$ for singular $\mu$ is intimately related to the concept of \emph{scales}:
	
	\begin{mydef}
		Let $\mu$ be a singular cardinal with countable cofinality. Let $(\mu_n)_{n\in\omega}$ be an increasing sequence of regular cardinals converging to $\mu$. A \emph{$(\mu^+,(\mu_n)_{n\in\omega})$-scale} is a sequence $(f_{\alpha})_{\alpha<\mu^+}$ such that:
		\begin{enumerate}
			\item For all $\alpha<\mu^+$, $f_{\alpha}\in\prod_{n\in\omega}\mu_n$;
			\item For all $\alpha<\beta<\mu^+$, $f_{\alpha}<^*f_{\beta}$, i.e. there is $k\in\omega$ such that $f_{\alpha}(n)<f_{\beta}(n)$ for all $n\geq k$;
			\item For all $g\in\prod_{n\in\omega}\mu_n$ there is $\alpha<\mu^+$ such that $g<^*f_{\alpha}$.
		\end{enumerate}
		
		Given a $(\mu^+,(\mu_n)_{n\in\omega})$-scale $(f_{\alpha})_{\alpha<\mu^+}$, we say that an ordinal $\gamma<\mu^+$ is \emph{good for $(f_{\alpha})_{\alpha<\mu^+}$} if there exists an unbounded $A\subseteq\gamma$ and $k\in\omega$ such that for all $n\geq k$, the sequence $(f_{\alpha}(n))_{\alpha\in A}$ is strictly increasing. $\gamma<\mu^+$ is \emph{bad for $(f_{\alpha})_{\alpha<\mu^+}$} if it is not good for $(f_{\alpha})_{\alpha<\mu^+}$.
	\end{mydef}
	
	Given a $(\mu^+,(\mu_n)_{n\in\omega})$-scale $\vec{f}$ and a set $S\subseteq\mu^+$ in $I[\mu^+]$, Shelah showed that there is a club subset $C\subseteq\mu^+$ such that every $\gamma\in S\cap C$ is good for $\vec{f}$. On the other hand, answering an old question due to Shelah, Maxwell Levine and the author showed in \cite{JakobLevineFailureApproachability} that it is consistent to have a $(\aleph_{\omega+1},(\aleph_n)_{n\in\omega})$-scale $\vec{f}$ such that club many points in $\aleph_{\omega+1}$ are good for $\vec{f}$, but $\aleph_{\omega+1}\notin I[\aleph_{\omega+1}]$. In our model, there will also be many instances of this phenomenon: By a folklore result\footnote{see \url{https://mathoverflow.net/q/296225}}, whenever $\vec{f}$ is a $(\mu^+,(\mu_n)_{n\in\omega})$-scale, there is a club subset $C\subseteq\mu^+$ such that any $\gamma\in C\cap E_{\geq\aleph_3}^{\mu^+}$ is good for $\vec{f}$. Despite this, in our model, this set will clearly be outside of $I[\mu^+]$.
	
	The approachability ideal can be defined for any cardinal $\mu$ but its behavior depends heavily on whether $\mu$ is regular or singular: If $\mu$ is regular, $E_{<\mu}^{\mu^+}\in I[\mu^+]$ (see \cite[Lemma 4.4]{ShelahApproachability}) and so $I[\mu^+]$ is completely determined by the membership of stationary subsets of $E_{\mu}^{\mu^+}$. On the other hand, if $\mu$ is singular, the only known results are that $E_{\leq\cf(\mu)}^{\mu^+}$ is an element of $I[\mu^+]$ and for any regular uncountable cardinal $\delta<\mu$, $I[\mu^+]$ contains a stationary subset of $E_{\delta}^{\mu^+}$ (see \cite[Theorem 9.2]{CummingsNotes}). As we will show, this result is sharp for successors of singular cardinals of countable cofinality.
	
	Another difference is in the way the failure of $\AP_{\mu}$ is commonly forced: For regular $\mu$ there is a natural correspondence between the non-approachability of points and the ``approximation property'' which occurs implicitly in work of Mitchell (see \cite{MitchellTreeProp}) and was explicitly defined by Hamkins (see \cite{HamkinsExtApprox}). Due to this, models where $\AP_{\mu}$ fails are often obtained by employing variants of Mitchell forcing or side condition posets. On the other hand, the failure of $\AP_{\mu}$ for $\mu$ singular is most commonly obtained by taking the ``natural failure'' of $\AP_{\mu}$ above supercompact cardinals and ``moving it down''. The most direct path to this concept comes by considering a characterization of approachability using colorings (this idea is also due to Shelah, see \cite{ShelahSuccSingCard}). We use the presentation from \cite{EisworthHandbook}.
	
	\begin{mydef}
		Let $\mu$ be a singular cardinal and $\lambda=\mu^+$. Let $d\colon[\lambda]^2\to\cf(\mu)$ be a coloring. We view $d$ as a function on the ascending pairs of elements of $\lambda$.
		\begin{enumerate}
			\item $d$ is \emph{subadditive} if for all $\alpha<\beta<\gamma<\lambda$,
			$$d(\alpha,\gamma)\leq\max\{d(\alpha,\beta),d(\beta,\gamma)\};$$
			\item $d$ is \emph{normal} if for all $\alpha\in\lambda$ and $i<\cf(\mu)$,
			$$|\{\beta<\alpha\;|\;d(\beta,\alpha)\leq i\}|<\mu;$$
			\item An ordinal $\alpha\leq\lambda$ is \emph{$d$-approachable} if there exists an unbounded set $A\subseteq\alpha$ such for all $\beta\in A$,
			$$\sup\{d(\beta,\alpha)\;|\;\beta\in A\cap\alpha\}<\cf(\mu);$$
			We let $S(d)$ consist of all $\alpha<\lambda$ which are $d$-approachable.
		\end{enumerate}
	\end{mydef}
	
	As Shelah shows in \cite{ShelahSuccSingCard}, whenever $\mu$ is singular, there is a normal subadditive coloring $d\colon[\mu^+]\to\cf(\mu)$.
	
	It is clear that $\lambda$ itself is never $d$-approachable with respect to any normal subadditive coloring $d\colon[\lambda]^2\to\cf(\mu)$. However, we have chosen to include it in the definition in order to be able to say ``$\lambda$ does not become $d$-approachable in $W$'', where $W$ is some outer model.
	
	$d$-approachability can be simplified as follows (see \cite[Corollary 3.27]{EisworthHandbook}):
	
	\begin{myfact}
		Let $\mu$ be a singular cardinal and $d\colon[\mu^+]^2\to\cf(\mu)$ a normal, subadditive coloring. An ordinal $\alpha<\mu^+$ is $d$-approachable if and only if $\cf(\alpha)\leq\cf(\mu)$ or there is a cofinal set $A\subseteq\alpha$ of minimal order-type and $i<\cf(\mu)$ such that for all $\beta<\gamma$, both in $A$, $d(\beta,\gamma)\leq i$.
	\end{myfact}
	
	In the case where the singular cardinal $\mu$ under consideration is a strong limit, the previous concepts are connected as follows:
	
	\begin{enumerate}
		\item The approachability ideal $I[\mu^+]$ is equal to the \emph{weak approachability ideal} $I[\mu^+,\mu]$. This ideal is defined almost the same way as $I[\mu^+]$ but instead of requiring $A\cap\beta\in\{x_i\;|\;i<\gamma\}$ for every $\beta<\gamma$ we merely require each $A\cap\beta$ to be a subset of $x_i$ for some $i<\gamma$.
		\item The approachability ideal is generated by a single set modulo the nonstationary ideal, i.e. there is $A\in I[\mu^+]$ such that for any $S\subseteq\mu^+$, $S\in I[\mu^+]$ if and only if $S\smallsetminus A$ is nonstationary.
	\end{enumerate}
	
	Combining this with the connection between $I[\mu^+,\mu]$ and $d$-approachability, we have the following:
	
	\begin{myfact}[{\cite[Corollary 3.35]{EisworthHandbook}}]\label{ColoringCanonical}
		Suppose that $\mu$ is a strong limit singular cardinal and $d\colon[\mu^+]^2\to\cf(\mu)$ is a normal, subadditive coloring. Then $S(d)\cup E_{\leq\cf(\mu)}^{\mu^+}$ generates $I[\mu^+]$ modulo the nonstationary ideal.
	\end{myfact}
	
	We finish this small exposition with the following helpful fact which states that whenever $\alpha$ is $d$-approachable, any cofinal subset of $\alpha$ can be refined to another cofinal subset on which $d$ is bounded:
	
	\begin{myfact}[{\cite[Remark 28]{ShelahSuccSingCard}}]\label{FactRefinement}
		Let $\mu$ be a singular cardinal and $d\colon[\mu^+]^2\to\cf(\mu)$ a subadditive coloring. Let $\alpha\in\mu^+$ with $\cf(\alpha)>\cf(\mu)$ and suppose that $\alpha$ is $d$-approachable. Then whenever $A\subseteq\alpha$ is unbounded, there is $B\subseteq\alpha$ unbounded and $i<\cf(\mu)$ such that whenever $\alpha<\beta$, both in $B$, $d(\alpha,\beta)\leq i$.
	\end{myfact}
	
	We now turn to supercompact cardinals. We use the following idea of Magidor (see \cite{MagidorSuperCompact}; the terminology is borrowed from \cite{MohamVelicGuessingModelsApproach}):
	
	\begin{mydef}
		Let $\kappa\leq\Theta$ be regular cardinals. $M\prec H(\Theta)$ is a \emph{$\kappa$-Magidor model} if the following holds:
		\begin{enumerate}
			\item $M\cap\kappa\in\kappa$,
			\item whenever $B\subseteq M$ is such that $B\subseteq C\in M$, there is $A\in M$ such that $A\cap M=B$.
		\end{enumerate}
	\end{mydef}
	
	Magidor showed that the supercompactness of a cardinal $\kappa$ is equivalent to the existence of sufficiently many $\kappa$-Magidor models. We only use one direction which is as follows:
	
	\begin{mysen}[Magidor]\label{MagModExists}
		Let $\kappa$ be a supercompact cardinal and $\Theta\geq\kappa$ regular. Then whenever $x\in H(\Theta)$ there is a $\kappa$-Magidor model $M\prec H(\Theta)$ with $x\in M$.
	\end{mysen}
	
	\begin{proof}
		Let $\Theta$ and $x$ be given. Let $j\colon V\to M$ be a $|H^V(\Theta)|$-supercompact embedding with critical point $\kappa$. Then $j(x)\in j[H^V(\Theta)]\prec H^M(j(\Theta))$. We verify that $j[H^V(\Theta)]$ is a $j(\kappa)$-Magidor model. Clearly $j[H^V(\Theta)]\cap j(\kappa)=\kappa\in j(\kappa)$. Let $B\subseteq j[H^V(\Theta)]$, $B\subseteq C\in j[H^V(\Theta)]$. Then $B=j[A]$ where $A\subseteq H^V(\Theta)$. Furthermore, $A\in H^V(\Theta)$, since $B$ is a subset of an element of $j[H^V(\Theta)]$. It follows that $B=j(A)\cap j[H^V(\Theta)]$.
		
		It follows by elementarity that there is a $\kappa$-Magidor model $M\prec H(\Theta)$ with $x\in M$.
	\end{proof}
	
	We will also make use of Laver functions in order to obtain indestructibility results:
	
	\begin{mydef}
		Let $\kappa$ be a supercompact cardinal. $l\colon\kappa\to V_{\kappa}$ is a \emph{Laver function} if for any $x$ and any sufficiently large $\lambda$ there is a $\lambda$-supercompact embedding $j\colon V\to M$ with critical point $\kappa$ such that $j(l)(\kappa)=x$.
	\end{mydef}
	
	Laver showed (see \cite{LaverIndestruct}) that any supercompact cardinal carries a Laver function:
	
	\begin{mysen}
		Let $\kappa$ be supercompact. Then there is a Laver function $j\colon V\to M$.
	\end{mysen}
	
	Since elementary embeddings and transitive collapses are ``dual'' to each other, the following lemma is not surprising. Recall that for any extensional well-founded structure $(M,E)$, the Mostowski-collapse $\pi_M\colon(M,E)\to(N,\in)$ is defined inductively by $\pi_M(x):=\{\pi_M(y)\;|\;y\in M\wedge yEx\}$.
	
	\begin{mylem}\label{SmallLaver}
		Let $\kappa$ be a supercompact cardinal and $\Theta\geq\kappa$ regular. Let $l\colon\kappa\to V_{\kappa}$ be a Laver function. Then whenever $x,y\in H(\Theta)$ there is a $\kappa$-Magidor model $M\prec H(\Theta)$ such that $x,y\in M$ and $l(M\cap\kappa)=\pi_M(y)$.
	\end{mylem}
	
	\begin{proof}
		Let $\Theta$ and $x,y$ be given. Let $j\colon V\to M$ be a sufficiently supercompact embedding with critical point $\kappa$ such that $j(l)(\kappa)=y$. As before, the pointwise image $j[H^V(\Theta)]\prec H^M(j(\Theta))$ is a $j(\kappa)$-Magidor model. One shows by induction that $\pi_{j[H^V(\Theta)]}(j(a))=a$ whenever $a\in H^V(\Theta)$, which shows the statement (for more details, see \cite[Lemma 5.4.14]{JakobPhD})
	\end{proof}
	
	Lastly, there is the following connection between Magidor models and normal, subadditive colorings:
	
	\begin{mylem}\label{InducedColoring}
		Let $\kappa$ be a supercompact cardinal, let $\delta\geq\kappa$ be a singular cardinal such that $\cf(\delta)<\kappa$ and let $\Theta>\delta$ be regular. Let $d\colon[\delta^+]^2\to\cf(\delta)$ be normal and subadditive and let $M\prec H(\Theta)$ be a $\kappa$-Magidor model with $d,\delta\in M$. Then the following holds:
		\begin{enumerate}
			\item Letting $\rho:=\sup(M\cap\delta^+)$, $\cf(\rho)$ is the successor of a singular cardinal $\mho$ with cofinality $\cf(\delta)$.
			\item There is an unbounded subset $A\subseteq\rho$ (we can choose $M\cap\delta^+$) such that whenever $\alpha\in A$ and $i<\cf(\delta)$,
			$$|\{\beta\in A\cap\alpha\;|\;d(\beta,\alpha)\leq i\}|<\mho.$$
		\end{enumerate}
	\end{mylem}
	
	\begin{proof}
		For (1), we first note that whenever $\mu\in M$ is a cardinal, $\otp(M\cap\mu)$ is a cardinal as well: Otherwise, there is $\alpha<\otp(M\cap\mu)$ and a well-order $X\subseteq\alpha\times\alpha$ with order-type $\otp(M\cap\mu)$. It follows that $\alpha=\otp(M\cap\nu)$ for some $\nu<\mu$, $\nu\in M$ and there exists $B\subseteq M\cap\nu\times M\cap\nu$ with order-type $\otp(M\cap\mu)$. Let $A\in M$ be such that $A\cap M=B$. Then by elementarity $A\subseteq\nu\times\nu$ and $\otp(M)=\mu$. But this contradicts the fact that $\mu$ is a cardinal.
		
		Furthermore, we know that $\sup(M\cap\delta)=\delta$, since $\cf(\delta)<\kappa$. It follows that $\otp(M\cap\delta)$ is a cardinal and has cofinality $\cf(\delta)$ (since, whenever $f\in M$ is a cofinal function from $\cf(\delta)$ into $\delta$, $f\uhr (M\cap\cf(\delta))=f$ is a cofinal function from $\cf(\delta)$ into $M\cap\delta$). It follows that $\otp(M\cap\delta^+)=\otp(M\cap\delta)^+$, so $\otp(M\cap\delta^+)$ (which equals $\cf(\sup(M\cap\delta^+))$) is the successor of a singular cardinal with cofinality $\cf(\delta)$. We let $\mho:=\otp(M\cap\delta)$.
		
		For (2), as claimed, we choose $A:=M\cap\delta^+$. Let $\alpha\in A$ and $i<\cf(\delta)$. Then $A_{\alpha}:=\{\beta\in A\cap\alpha\;|\;d(\beta,\alpha)\leq i\}\subseteq M\cap\alpha$ and so there is $B\in M$ with $B\cap M=A_{\alpha}$. It follows by elementarity that $B=\{\beta\in\alpha\;|\;d(\beta,\alpha)\leq i\}$. Again by elementarity, $|B|<\delta$, so there is (in $M$) some $\mu<\delta$ and a surjection $f\colon\mu\to B$. Then $f\uhr(\mu\cap M)$ is a surjection from $(\mu\cap M)$ onto $B\cap M=A_{\alpha}$. Since $\mu<\delta$ and $\otp(M\cap\delta)$ is a cardinal, it follows that $\otp(M\cap\mu)<\otp(M\cap\delta)=\mho$ and we are done.
	\end{proof}
	
	It follows from (2) that $d$ induces a normal, subadditive coloring on $\mho^+$ in the following way: Let $f\colon\mho^+\to M\cap\delta^+$ be order-preserving. We let $e\colon[\mho^+]^2\to\cf(\mho)=\cf(\delta)$ be defined by $e(\alpha,\beta):=d(f(\alpha),f(\beta))$. Then $e$ is a normal, subadditive coloring on $\mho^+$. Additionally, in any outer model, $\rho$ is $d$-approachable if and only if $\mho^+$ is $e$-approachable (by Fact \ref{FactRefinement}). This shows that the problem of obtaining non-approachable points of arbitrary cofinality reduces to the question of whether or not, given a singular cardinal $\delta$ and a regular cardinal $\mu<\delta$, it is possible to collapse $|\delta^+|=\mu$ without making $\delta^+$ $d$-approachable with respect to any normal subadditive coloring $d\colon[\delta^+]^2\to\cf(\delta)$. This problem was answered by the author and Levine in \cite{JakobLevineFailureApproachability}. In Section 4, we will introduce a poset with the same crucial property which can be defined using (in practice) much weaker hypotheses which allows us to iterate the construction and prove our main theorem.
	
	So in summary, we have shown:
	
	\begin{mylem}
		Let $\kappa$ be a supercompact cardinal. let $\mu>\kappa$ be singular with $\cf(\mu)<\kappa$. Then whenever $d\colon[\mu^+]^2\to\cf(\mu)$ is a normal, subadditive coloring, there are stationarily many $\alpha\in\mu^+$ which are not $d$-approachable. In particular, $\mu^+\notin I[\mu^+]$.
	\end{mylem}
	
	Lastly, we prove that Magidor models have transitive collapses which reflect many set-theoretical statements. This mirrors the fact that for all ``bounded'' statements $\phi$, one can find a sufficiently supercompact embedding $j\colon V\to M$ such that $V\models\phi\leftrightarrow M\models\phi$ due to the closure of $M$:
	
	\begin{mylem}\label{MagidorModelElementary}
		Let $\kappa$ be a supercompact cardinal, let $\Theta>\kappa$ be regular and let $M\prec H(\Theta)$ be a $\kappa$-Magidor model. Let $\pi\colon M\to N$ be the transitive collapse of $M$. Whenever $\delta\in M$ is regular, $H(\pi(\delta))\subseteq N$.
	\end{mylem}
	
	\begin{proof}
		Since every element of $H(\pi(\delta))$ can be coded as a subset of $\pi(\delta)$, it suffices to show that $\mathcal{P}(\pi(\delta))\subseteq N$. To this end, let $x\subseteq\pi(\delta)$. Then $\pi^{-1}[x]\subseteq\delta\cap M$ and so there is $A\in M$ with $A\cap M=\pi^{-1}[x]$. Then $\pi(A)=\pi[A\cap M]=x\in N$.
	\end{proof}
	
	So this means that whenever $M\prec H(\Theta)$ is a $\kappa$-Magidor model, the transitive collapse $N$ is elementary in $V$ for every statement which is verifiable by some large enough $H(\pi(\Theta'))$ for $\Theta'\in H(\Theta)$.
	
	\subsection{Forcing}\hfill
	
	In this subsection, we introduce some lesser known definitions and facts regarding the technique of forcing. Our notation is standard. To reduce confusion, we write ${<}\,\kappa$-closed (directed closed; distributive; etc.). We follow the convention that filters are upwards closed, so that if $p\leq q$, $p$ forces more than $q$. Given a poset $\dP$, we let $\Gamma_{\dP}$ be the canonical $\dP$-name for the $\dP$-generic filter, i.e. $\Gamma_{\dP}:=\{(p,\check{p})\;|\;p\in\dP\}$.
	
	We start with the separative quotient:
	
	\begin{mydef}
		Let $\dP$ be a poset. On $\dP$, define the equivalence relation $\sim$ by $p\sim q$ if and only if $\forall z(z\perp p\leftrightarrow z\perp q)$. Define $\preceq$ on $\dP/\sim$ by $[p]\preceq[q]$ if and only if $\forall r\leq p(r||q)$. Then $(\dP/\sim,\preceq)$ is called the \emph{separative quotient} of $\dP$ and $h\colon\dP\to(\dP/\sim)$, $h(p)=[p]$ is the quotient mapping.
	\end{mydef}
	
	It is easy to verify that $p\leq q\to h(p)\preceq h(q)$ and $p,q$ are compatible in $\dP$ if and only if they are compatible in $\dP/\sim$. Additionally, one sees easily that $\dP$ and $\dP/\sim$ are forcing equivalent.
	
	We will also make use of projection analyses. In the literature, there are two common ways of defining projections, one due to Abraham (see \cite[Page 335]{AbrahamHandbook}) and a slightly more general one due to Cummings (see \cite[Definition 5.2]{CummingsHandbook}). In this work, we will use the definition due to Cummings, since the results it provides are sufficient for us.\footnote{See \url{https://mathoverflow.net/q/466836} for a discussion regarding the discrepancy.}
	
	\begin{mydef}
		Let $\dP$ and $\dQ$ be posets. A function $\pi\colon\dP\to\dQ$ is a \emph{projection} if the following holds:
		\begin{enumerate}
			\item $\pi(1_{\dP})=1_{\dQ}$.
			\item Whenever $p\in\dP$ and $q\leq\pi(p)$, there is $p'\leq p$ such that $\pi(p')\leq q$.
		\end{enumerate}
	\end{mydef}
	
	The existence of a projection from $\dP$ onto $\dQ$ implies that, after forcing with $\dQ$, the resulting model can once again be extended by forcing to a model which is a forcing extension by $\dP$. This is done using the \emph{quotient forcing}:
	
	\begin{mydef}
		Let $\dP$ and $\dQ$ be poset, $\pi\colon\dP\to\dQ$ a projection. Let $G$ be a $\dQ$-generic filter. In $V[G]$, we let $\dP/G$ consist of all those $p\in\dP$ such that $\pi(p)\in G$, ordered as a suborder of $\dP$. We call $\dP/G$ the \emph{quotient forcing} of $\dP$ and $G$. We also let $\dP/\dQ$ be a $\dQ$-name for $\dP/\Gamma_{\dQ}$ and refer to $\dP/\dQ$ as the \emph{quotient forcing} of $\dP$ and $\dQ$.
	\end{mydef}
	
	The quotient forcing actually works as intended:
	
	\begin{mylem}
		Let $\dP$ and $\dQ$ be posets, $\pi\colon\dP\to\dQ$ a projection. Let $G$ be a $\dQ$-generic filter and let $H$ be $\dP/G$-generic over $V[G]$. Then $H$ is $\dP$-generic over $V$ and $H=\pi[G]$. In particular, $V[G][H]=V[H]$.
	\end{mylem}
	
	It follows that, after forcing with $\dQ$, we can perform further forcing to reach a forcing extension by $\dP$. This allows us to obtain regularity properties for a poset $\dQ$ by finding a well-behaved poset $\dP$ and a projection $\pi\colon\dP\to\dQ$. It is clear that in this situation, if $\dP$ does not collapse certain cardinals or add sequences of ordinals of a certain length, neither does $\dQ$, since any forcing extension by $\dQ$ is contained in a forcing extension by $\dP$.
	
	For iterations, one can easily find projections by using the \emph{termspace forcing}, an idea due to Laver:
	
	\begin{mydef}
		Let $\dP$ be a poset and $\dot{\dQ}$ a $\dP$-name for a poset. Then the termspace forcing $(\dT(\dP,\dot{\dQ}),\preceq)$ is defined as follows: Conditions are $\dP$-names $\dot{q}$ for elements of $\dot{\dQ}$ (i.e. $1_{\dP}\Vdash\dot{q}\in\dot{\dQ}$), ordered by $\dot{q}'\preceq\dot{q}$ if and only if $1_{\dP}\Vdash\dot{q}'\leq_{\dot{\dQ}}\dot{q}$.
	\end{mydef}
	
	Using standard arguments on names, one shows:
	
	\begin{mylem}\label{Projection}
		Let $\dP$ be a poset and $\dot{\dQ}$ a $\dP$-name for a poset. The identity function is a projection from $\dP\times\dT(\dP,\dot{\dQ})$ onto $\dP*\dot{\dQ}$.
	\end{mylem}
	
	We will actually use an extension of the previous result (which is nonetheless shown using almost the same arguments):
	
	\begin{mylem}[Foreman \cite{ForemanSatIdeal}]\label{TermspaceIteration}
		Let $(\dP_{\alpha},\dot{\dQ}_{\alpha})_{\alpha<\kappa}$ be a full support (Easton support) iteration. Let $\dT:=\prod_{\alpha<\kappa}\dT(\dP_{\alpha},\dot{\dQ}_{\alpha})$, where the product is taken with full support (Easton support). Then the identity is a projection from $\dT$ onto the inverse (Easton) limit of $(\dP_{\alpha},\dot{\dQ}_{\alpha})_{\alpha<\kappa}$.
	\end{mylem}
	
	\begin{proof}
		We do the proof in the full support case, but the proof in the Easton support case is entirely the same.
		
		Let $\dP:=(\dP,\leq_{\dP})$ be the inverse limit of $(\dP_{\alpha},\dot{\dQ}_{\alpha})_{\alpha<\kappa}$. Denote by $\leq_{\dT}$ the ordering on $\dT$. It is clear that the identity is a function from $\dT$ onto $\dP$ and preserves the ordering. Let $p\in\dP$ and let $t\leq_{\dP}p$. We want to find $p'\leq_{\dT}p$ with $p'\leq_{\dP}t$. By induction on $\alpha$, let $p'(\alpha)$ be a $\dP_{\alpha}$-name for an element of $\dot{\dQ}_{\alpha}$ forced by $p'\uhr\alpha$ to be equal to $t(\alpha)$ and by conditions incompatible with $p'\uhr\alpha$ to be equal to $p(\alpha)$. Since $p'(\alpha)$ is in any case forced to be below $p(\alpha)$, $p'\leq_{\dT}p$. Since $p'(\alpha)$ is always forced by $p'\uhr\alpha$ to be equal to $t(\alpha)$, $p'\leq_{\dP}t$.
	\end{proof}
	
	The proof of the main theorem is done by successively forcing, for a regular cardinal $\mu$, that for every singular $\delta>\mu$ the set $E_{\mu}^{\delta^+}$ is not in $I[\delta^+]$. In order to prove that the desired conclusion holds, we must therefore show that the tail of the iteration does not destroy the stationarity of the already added sets that lie outside of $I[\delta^+]$. We will do so by proving that a certain elementary embedding \emph{lifts through the forcing extension}, which requires us to show that the tail forcing is sufficiently directed-closed:
	
	\begin{mydef}
		Let $\dP$ be a poset. $X\subseteq\dP$ is \emph{directed} if whenever $p,q\in X$, there is $r\in X$ such that $r\leq p,q$. If $\mu$ is a cardinal, $\dP$ is \emph{${<}\,\mu$-directed closed} if whenever $X\subseteq\dP$ is directed and $|X|<\mu$, there exists $p\in\dP$ such that $p\leq q$ for every $q\in X$.
	\end{mydef}
	
	In the proof of Lemma \ref{LIPForcingLemma}, there is a single point where the established notion of directed closure seems to be insufficient, unless the poset under consideration is separative. Since the closure of a poset might fail to be inherited by its separative quotient\footnote{see \url{https://mathoverflow.net/q/11505}}, we are defining the following strengthening of directed closure instead and using it to prove that the separative quotient of our considered poset is directed-closed.
	
	\begin{mydef}
		Let $\dP$ be a poset. $X\subseteq\dP$ is \emph{weakly directed} if whenever $p,q\in X$, there is $r\in X$ such that $r\Vdash\check{p}\in\Gamma_{\dP}\wedge\check{q}\in\Gamma_{\dP}$. If $\mu$ is a cardinal, $\dP$ is \emph{strongly ${<}\,\mu$-directed closed} if whenever $X\subseteq\dP$ is weakly directed and $|X|<\mu$, there exists $p\in\dP$ such that $p\leq q$ for every $q\in X$.
	\end{mydef}
	
	We note that $r\Vdash\check{q}\in\Gamma_{\dP}$ if and only if $q$ is compatible with every extension of $r$: On one hand, if there is $s\leq r$ which is incompatible with $q$, any generic filter containing $s$ contains $r$ but not $q$. On the other hand, if $q$ is compatible with every extension of $r$ and $G$ is $\dP$-generic containing $r$, the upward closure of $G\cup\{q\}$ is a filter and thus equal to $G$ by genericity. Hence, $q\in G$ and since $G$ was arbitrary, $r\Vdash\check{q}\in\Gamma_{\dP}$.
	
	Clearly, for separative forcings $\dP$, $X\subseteq\dP$ is weakly directed if and only if it is directed, since $p\Vdash\check{q}\in\Gamma_{\dP}$ if and only if $p\leq q$. One might think that the notion of strong directed closure could be made redundant by simply working with the separative closures of our partial orders but in general the separative closure of a partial order must not share the same closure properties as the original poset.
	
	Strong directed closure however is inherited by the separative quotient:
	
	\begin{mylem}
		Let $\dP$ be a poset and $\mu$ a cardinal. If $\dP$ is strongly ${<}\,\mu$-directed closed, then the separative quotient of $\dP$ is ${<}\,\mu$-directed closed.
	\end{mylem}
	
	\begin{proof}
		Let $h\colon\dP\to\dP/\sim$ be the quotient mapping and let $X\subseteq\dP/\sim$ be directed, $|X|<\mu$. Let $\hat{X}$ be such that $X=\{[p]_{\sim}\;|\;p\in\hat{X}\}$. Then $|\hat{X}|=|X|$ and $\hat{X}$ is weakly directed: Let $p,q\in\hat{X}$. Since $X$ is directed, there is $r\in\hat{X}$ such that $[r]_{\sim}\preceq[p]_{\sim},[q]_{\sim}$. It follows that any extension of $r$ is compatible with $p$ and $q$. Ergo, $[r]_{\sim}\Vdash[\check{p}]_{\sim}\in\Gamma_{\dP/\sim}\wedge[\check{q}]_{\sim}\in\Gamma_{\dP/\sim}$. Since $\dP$ is strongly ${<}\,\mu$-directed closed, there exists $p\in \dP$ such that $p\leq q$ for every $q\in\hat{X}$. Then whenever $[q]_{\sim}\in X$, $q\in\hat{X}$, so $h(p)\leq h(q)=[q]_{\sim}$. Ergo $h(p)$ is a lower bound of $X$.
	\end{proof}
	
	It is easy to see that strong directed closure is productive. Additionally, the usual proof of the iterability of directed closure goes through for strong directed closure as well. Note that, in general, this does not automatically follow from the iterability of directed closure, since two-step iterations typically fail to be separative.
	
	\begin{mylem}[{\cite[Lemma 21.7]{JechSetTheory}}]\label{StrongDirIter}
		Let $\mu$ be a regular cardinal.
		\begin{enumerate}
			\item If $\dP$ is strongly ${<}\,\mu$-directed closed and $\Vdash_{\dP}$``$\dot{\dQ}$ is strongly ${<}\,\check{\mu}$-directed closed'', then $\dP*\dot{\dQ}$ is strongly ${<}\,\mu$-directed closed.
			\item If $\cf(\alpha)>\mu$ and $\dP_{\alpha}$ is a direct limit of $(\dP_{\beta})_{\beta<\alpha}$ such that each $\dP_{\beta}$ is strongly ${<}\,\mu$-directed closed, then $\dP_{\alpha}$ is ${<}\,\mu$-directed closed.
			\item If $\dP_{\alpha}$ is the limit of a forcing iteration $(\dP_{\beta},\dot{\dQ}_{\beta})_{\beta<\alpha}$ such that for each limit ordinal $\beta\leq\alpha$, $\dP_{\beta}$ is either a direct limit or an inverse limit and it is an inverse limit if $\cf(\beta)\leq\mu$, then $\dP_{\alpha}$ is $\mu$-directed closed.
		\end{enumerate}
	\end{mylem}
	
	\begin{proof}
		We show just the first statement, the other ones follow just as in the original source. So let $\dP$ and $\dot{\dQ}$ be as given and let $X\subseteq\dP*\dot{\dQ}$ be weakly directed, $X=\{(p_{\alpha},\dot{q}_{\alpha})\;|\;\alpha<\gamma\}$ ($\gamma<\mu$). It follows that $\{p_{\alpha}\;|\;\alpha<\gamma\}$ is weakly directed (if $(p_0,\dot{q}_0)\Vdash(p_1,\dot{q}_1)\in\Gamma_{\dP*\dot{\dQ}}$, then in particular $p_0\Vdash\check{p}_1\in\Gamma_{\dP}$) and so there is $p$ such that $p\leq p_{\alpha}$ for all $\alpha<\gamma$. Let $G$ be $\dP$-generic over $V$ with $p\in G$. Then $\{\dot{q}_{\alpha}^G\;|\;\alpha<\gamma\}$ is a weakly directed subset of $\dot{\dQ}^G$ (if $(p_0,\dot{q}_0)\Vdash(p_1,\dot{q}_1)\in\Gamma_{\dP*\dot{\dQ}}$, then $p_0\Vdash\dot{q}_0\Vdash\dot{q}_1\in\Gamma_{\dot{\dQ}}$ and so $\dot{q}_0^G\Vdash\dot{q}_1^G\in\Gamma_{\dot{\dQ}^G}$). Ergo there is $\dot{q}^G\in\dot{\dQ}^G$ such that $\dot{q}^G\leq\dot{q}_{\alpha}^G$ for all $\alpha<\gamma$. By the maximum principle we can fix $\dot{q}$ such that $p\Vdash\dot{q}\leq\dot{q}_{\alpha}$ for all $\alpha<\gamma$. Then $(p,\dot{q})\leq(p_{\alpha},\dot{q}_{\alpha})$ for all $\alpha<\gamma$.
	\end{proof}
	
	Another proof which adapts to this new definition is that strong directed closure is inherited by the termspace forcing. Note that, again, this is not automatic, since in general the termspace forcing is not separative.
	
	\begin{mylem}\label{TermspaceStrongDirect}
		Let $\mu$ be a regular cardinal, let $\dP$ be a poset and $\dot{\dQ}$ a $\dP$-name for a poset such that $\Vdash_{\dP}$ `` $\dot{\dQ}$ is strongly ${<}\,\check{\mu}$-directed closed''. Then $\dT(\dP,\dot{\dQ})$ is strongly ${<}\,\mu$-directed closed.
	\end{mylem}
	
	\begin{proof}
		Let $\{\dot{q}_{\alpha}\;|\;\alpha<\gamma\}$ (where $\gamma<\mu$) be a weakly directed subset of $\dT(\dP,\dot{\dQ})$. We want to show that whenever $G$ is $\dP$-generic, $\{\dot{q}_{\alpha}^G\;|\;\alpha<\gamma\}$ is a weakly directed subset of $\dot{\dQ}^G$. To this end, let $G$ be $\dP$-generic and let $\alpha_0,\alpha_1<\gamma$. There exists $\alpha_2<\gamma$ such that $\dot{q}_{\alpha_2}\Vdash\dot{q}_{\alpha_i}\in\Gamma_{\dT(\dP,\dot{\dQ})}$ for $i=0,1$. In particular, any extension of $\dot{q}_{\alpha_2}$ (in $\dT(\dP,\dot{\dQ})$) is compatible with $\dot{q}_{\alpha_i}$ for $i=0,1$. We will show that any extension of $\dot{q}_{\alpha_2}^G$ (in $\dot{\dQ}^G$) is compatible with $\dot{q}_{\alpha_i}^G$ for $i=0,1$.
		
		To this end, let $\dot{q}^G\leq_{\dot{\dQ}^G}\dot{q}_{\alpha_2}^G$. Then there is $p\in G$ which forces $\dot{q}\leq_{\dot{\dQ}}\dot{q}_{\alpha_2}$. Let $\dot{q}'$ be such that $p\Vdash\dot{q}=\dot{q}'$ and conditions incompatible with $p$ force $\dot{q}=\dot{q}_{\alpha_2}$. Then $\dot{q}'\preceq\dot{q}_{\alpha_2}$ (in $\dT(\dP,\dot{\dQ})$) and so $\dot{q}'$ is compatible with $\dot{q}_{\alpha_i}$ (again in $\dT(\dP,\dot{\dQ})$) for $i=0,1$. Let $\dot{q}^i\preceq\dot{q}',\dot{q}_{\alpha_i}$. Then $(\dot{q}^i)^G\leq_{\dot{\dQ}^G}\dot{q}_{\alpha_i}^G,(\dot{q}')^G$ but $(\dot{q}')^G=\dot{q}^G$, so $\dot{q}^G$ and $\dot{q}_{\alpha_i}^G$ are compatible for $i=0,1$.
		
		Now since $\dot{\dQ}$ is forced to be strongly ${<}\,\mu$-directed closed, $\dP$ forces that there is a lower bound $\dot{q}$ of $\{\dot{q}_{\alpha}\;|\;\alpha<\gamma\}$ in $\dot{\dQ}$. Therefore $\dot{q}$ is a lower bound of $\{\dot{q}_{\alpha}\;|\;\alpha<\gamma\}$ in $\dT(\dP,\dot{\dQ})$.
	\end{proof}
	
	\section{Iteration Theorems for Prikry-Type forcings}
	
	Due to technical reasons, most of our results will be obtained using Prikry-type forcings. In this section, we give definitions and prove results regarding iterations of these posets.
	
	\begin{mydef}
		Let $(\dP,\leq)$ be a partial order and let $\leq_0$ be another partial order on $\dP$ such that $\leq$ refines $\leq_0$. We say that $(\dP,\leq,\leq_0)$ is a \emph{Prikry-type forcing} if for every $p\in\dP$ and every statement $\sigma$ in the forcing language there is $q\leq_0p$ which decides $\sigma$. The latter statement is known as the \emph{Prikry property}.
	\end{mydef}
	
	For any Prikry-type forcing $(\dP,\leq,\leq_0)$, the order $\leq_0$ is only used to derive regularity properties. As such, we will abuse notation and refer to forcing ``with $\dP$'' as forcing with $(\dP,\leq)$. In the same way, we will let ``$\tau$ is a $\dP$-name'' refer to the fact that $\tau$ is a $(\dP,\leq)$-name.
	
	The first example of a Prikry-type forcing is the original poset used by Karel Prikry to give a measurable cardinal $\kappa$ cofinality $\omega$ in a forcing extension without collapsing any cardinals (see \cite{PrikryOriginal}): For $\kappa$ measurable and $U$ a normal measure over $\kappa$, the forcing $\dP(U)$ adds a sequence $(\nu_n)_{n\in\omega}$ which is cofinal in $\kappa$. In particular, $\dP(U)$ is not countably closed. However, there is an additional order $\leq^*$ on $\dP(U)$ such that $(\dP(U),\leq,\leq^*)$ is a Prikry-type forcing notion and $(\dP(U),\leq^*)$ is ${<}\,\kappa$-closed. This implies that $\dP(U)$ does not add any new bounded subsets of $\kappa$ and thus does not collapse cardinals below and including $\kappa$. Combined with the $\kappa^+$-cc of $\dP(U)$ we can see that no cardinals are collapsed even though the regularity of $\kappa$ is destroyed.
	
	Other examples include variants of Namba forcing. The original Namba forcing (see \cite{BukovskyCoplakovaMinimal}) shares some similarities with Prikry forcing. It adds a countable cofinal sequence to $\aleph_2$ without collapsing $\aleph_1$. If $\CH$ holds, it is even true that no new reals are added. However, in the case of Namba forcing, both of these regularity properties follow from abstract considerations (such as Shelah's \emph{$\dI$-condition}, see \cite[Chapter XI]{ShelahProperImproper}) which introduce complications regarding the preservation of larger cardinals and iterability. To remedy both of these problems, we will be using Namba forcings defined from particularly well-behaved ideals which function more similarly to supercompact Prikry forcing.
	
	We start with the problem of iterating Prikry-type forcings. This was largely solved by Magidor in his famous paper proving that it is consistent that the first strongly compact cardinal can also be the first measurable cardinal. He achieved this by successively singularizing every smaller measurable cardinal (see \cite{MagidorIdentityCrises}). We will be using iterations both with full support and with Easton support. The most important point that makes sure the iterations have the Prikry property is the restriction on where non-direct extensions can be used: For full support, only finitely many non-direct extensions are allowed while for Easton support, arbitrarily many non-direct extensions are allowed outside the previous support and only finitely many non-direct extensions are allowed inside the previous support.
	
	We begin with the definition and properties of full support iterations:
	
	\begin{mydef}
		Let $((\dP_{\alpha},\leq_{\alpha},\leq_{\alpha,0}),(\dot{\dQ}_{\alpha},\dot{\leq}_{\alpha},\dot{\leq}_{\alpha,0}))_{\alpha<\rho}$ be a sequence such that each $(\dP_{\alpha},\leq_{\alpha},\leq_{\alpha,0})$ is a Prikry-type poset and each $(\dot{\dQ}_{\alpha},\dot{\leq}_{\alpha},\dot{\leq}_{\alpha,0})$ is a $\dP_{\alpha}$-name for a Prikry-type poset. We define the statement ``$((\dP_{\alpha},\leq_{\alpha},\leq_{\alpha,0}),(\dot{\dQ}_{\alpha},\dot{\leq}_{\alpha},\dot{\leq}_{\alpha,0}))_{\alpha<\rho}$ is a full support Magidor iteration of Prikry-type forcings of length $\rho$'' by induction on $\rho$.
		
		$((\dP_{\alpha},\leq_{\alpha},\leq_{\alpha,0}),(\dot{\dQ}_{\alpha},\dot{\leq}_{\alpha},\dot{\leq}_{\alpha,0}))_{\alpha<\rho}$ is a full support Magidor iteration of Prikry-type forcings of length $\rho$ if $((\dP_{\alpha},\leq_{\alpha},\leq_{\alpha,0}),(\dot{\dQ}_{\alpha},\dot{\leq}_{\alpha},\dot{\leq}_{\alpha,0}))_{\alpha<\rho'}$ is a full support Magidor iteration of Prikry-type forcings of length $\rho'$ for every $\rho'<\rho$ and moreover:
		\begin{enumerate}
			\item If $\rho=\rho'+1$, then $(\dP_{\rho},\leq_{\rho})=(\dP_{\rho'},\leq_{\rho'})*(\dot{\dQ}_{\rho'},\dot{\leq}_{\rho'})$ and $(p',\dot{q}')\leq_{\rho,0}(p,\dot{q})$ if and only if $p'\leq_{\rho',0}p$ and $p'\Vdash\dot{q}'\dot{\leq}_{\rho',0}\dot{q}$.
			\item If $\rho$ is a limit, then $\dP_{\rho}$ consists of all functions $p$ on $\rho$ such that for all $\alpha<\rho$, $p\uhr\alpha\in\dP_{\alpha}$ and the following holds:
			\begin{enumerate}
				\item[(i)] $p'\leq_{\rho}p$ if and only if $p'\uhr\rho'\leq_{\rho'}p\uhr\rho'$ for every $\rho'<\rho$ and there exists a finite subset $b$ of $\rho$ such that whenever $\rho'\notin b$, then $p'\uhr\rho'\Vdash p'(\rho')\dot{\leq}_{\rho',0} p(\rho')$.
				\item[(ii)] $p'\leq_{\rho,0}p$ if and only if $p'\leq_{\rho}p$ and the set $b$ is empty.
			\end{enumerate} 
		\end{enumerate}
	\end{mydef}
	
	So we take the usual full support iteration with the caveat that we are only allowed to pass to a non-direct extension at merely finitely many coordinates. This is necessary in order to obtain that such iterations have the Prikry property:
	
	\begin{mylem}[{\cite[Lemma 6.2]{GitikHandbook}}]\label{PrikryIterFull}
		Let $((\dP_{\alpha},\leq_{\alpha},\leq_{\alpha,0}),(\dot{\dQ}_{\alpha},\dot{\leq}_{\alpha},\dot{\leq}_{\alpha,0}))_{\alpha<\rho}$ be a full-support Magidor iteration of Prikry-type forcings of length $\rho$. Then each $\dP_{\alpha}$ has the Prikry property.
	\end{mylem}
	
	It follows just as in Lemma \ref{StrongDirIter} that the following holds (since $\leq_{\alpha,0}$ refines $\leq_{\alpha}$ so that any $\leq_{\alpha,0}$-lower bound is also a $\leq_{\alpha}$-lower bound):
	
	\begin{mylem}\label{FullSupportClosure}
		Let $((\dP_{\alpha},\leq_{\alpha},\leq_{\alpha,0}),(\dot{\dQ}_{\alpha},\dot{\leq}_{\alpha},\dot{\leq}_{\alpha,0}))_{\alpha<\rho}$ be a full-support Magidor iteration of Prikry-type forcings of length $\rho$ and $\mu$ a cardinal. Assume that for each $\alpha<\rho$, $\dP_{\alpha}$ forces that $(\dot{\dQ}_{\alpha},\dot{\leq}_{\alpha,0})$ is strongly ${<}\,\check{\mu}$-directed closed. Then for each $\alpha<\rho$, $(\dP_{\alpha},\leq_{\alpha,0})$ is strongly ${<}\,\mu$-directed closed.
	\end{mylem}
	
	We now introduce the Easton-support Magidor iteration which has the following curious change: In the full-support iteration we are only allowed to take non-direct extensions finitely often. However, in the Easton-support iteration we are allowed to take arbitrarily many non-direct extensions provided that they occur outside of the support (inside the support, only finitely many extensions are allowed). As before, this is necessary in order to obtain the Prikry property.
	
	\begin{mydef}
		Let $((\dP_{\alpha},\leq_{\alpha},\leq_{\alpha,0}),(\dot{\dQ}_{\alpha},\dot{\leq}_{\alpha},\dot{\leq}_{\alpha,0}))_{\alpha<\rho}$ be a sequence such that each $(\dP_{\alpha},\leq_{\alpha},\leq_{\alpha,0})$ is a poset and each $(\dot{\dQ}_{\alpha},\dot{\leq}_{\alpha},\dot{\leq}_{\alpha,0})$ is a $\dP_{\alpha}$-name for a Prikry-type poset. We define the statement ``$((\dP_{\alpha},\leq_{\alpha},\leq_{\alpha,0}),(\dot{\dQ}_{\alpha},\dot{\leq}_{\alpha},\dot{\leq}_{\alpha,0}))_{\alpha<\rho}$ is an Easton-support Magidor iteration of Prikry-type forcings of length $\rho$'' by induction on $\rho$.
		
		$((\dP_{\alpha},\leq_{\alpha},\leq_{\alpha,0}),(\dot{\dQ}_{\alpha},\dot{\leq}_{\alpha},\dot{\leq}_{\alpha,0}))_{\alpha<\rho}$ is an Easton-support Magidor iteration of Prikry-type forcings of length $\rho$ if $((\dP_{\alpha},\leq_{\alpha},\leq_{\alpha,0}),(\dot{\dQ}_{\alpha},\dot{\leq}_{\alpha},\dot{\leq}_{\alpha,0}))_{\alpha<\rho'}$ is an Easton-support Magidor iteration of Prikry-type forcings of length $\rho'$ for every $\rho'<\rho$ and moreover:
		\begin{enumerate}
			\item If $\rho=\rho'+1$, then $(\dP_{\rho},\leq_{\rho})=(\dP_{\rho'},\leq_{\rho'})*(\dot{\dQ}_{\rho'},\dot{\leq}_{\rho'})$ and $(p',\dot{q}')\leq_{\rho,0}(p,\dot{q})$ if and only if $p'\leq_{\rho',0}p$ and $p'\Vdash\dot{q}'\dot{\leq}_{\rho',0}\dot{q}$.
			\item If $\rho$ is a limit, then $\dP_{\rho}$ consists of all functions $p$ on $\rho$ such that
			\begin{enumerate}
				\item For all $\alpha<\rho$, $p\uhr\alpha\in\dP_{\alpha}$,
				\item If $\rho$ is inaccessible and $|\dP_{\alpha}|<\rho$ for every $\alpha<\rho$, then there is some $\beta<\rho$ such that for all $\gamma\in(\beta,\rho)$, $p\uhr\gamma\Vdash p(\gamma)=1_{\dot{\dQ}_{\gamma}}$.
			\end{enumerate}
			and the following holds:
			\begin{enumerate}
				\item[(i)] $p'\leq_{\rho}p$ if and only if $p'\uhr\rho'\leq_{\rho'}p\uhr\rho'$ for every $\rho'<\rho$ and there exists a finite subset $b$ of $\rho$ such that whenever $\rho'\notin b$ and $p\uhr\rho'\not\Vdash p(\rho')=1_{\dot{\dQ}_{\rho'}}$, then $p'\uhr\rho'\Vdash p'(\rho')\dot{\leq}_{\rho',0} p(\rho')$.
				\item[(ii)] $p'\leq_{\rho,0}p$ if and only if $p'\leq_{\rho}p$ and the set $b$ is empty.
			\end{enumerate} 
		\end{enumerate}
	\end{mydef}
	
	And we have the following (see \cite[Section 6.3]{GitikHandbook} or \cite[Lemma 2.4]{JakobLevineFailureApproachability} for a more detailed proof):
	
	\begin{mylem}\label{PrikryIterEaston}
		Let $\rho$ be an ordinal. Assume that $((\dP_{\alpha},\leq_{\alpha},\leq_{\alpha,0}),(\dot{\dQ}_{\alpha},\dot{\leq}_{\alpha},\dot{\leq}_{\alpha,0}))_{\alpha<\rho}$ is an Easton-support Magidor iteration of Prikry-type forcings of length $\rho$. Then each $\dP_{\alpha}$ has the Prikry property.
	\end{mylem}
	
	And we have almost the same lemma as in the case of the full-support iteration with the obvious caveat that we are taking direct limits at inaccessible cardinals which are larger than the size of any initial segment of the iteration thus far.
	
	\begin{mylem}\label{DirectClosureEaston}
		Let $((\dP_{\alpha},\leq_{\alpha},\leq_{\alpha,0}),(\dot{\dQ}_{\alpha},\dot{\leq}_{\alpha},\dot{\leq}_{\alpha,0}))_{\alpha<\rho}$ be an Easton-support Magidor iteration of Prikry-type forcings and $\mu$ a cardinal such that $\mu<\delta$ whenever $\delta$ is inaccessible and $|\dP_{\alpha}|<\delta$ for all $\alpha<\delta$. If for all $\alpha<\rho$, $\dP_{\alpha}$ forces $(\dot{\dQ}_{\alpha},\dot{\leq}_{\alpha,0})$ to be strongly ${<}\,\check{\mu}$-directed closed, then for all $\alpha<\rho$, $(\dP_{\alpha},\leq_{\alpha,0})$ is strongly ${<}\,\mu$-directed closed.
	\end{mylem}
	
	It is clear that the Prikry property amounts to being able to decide ordinals ${<}\,2$ using direct extensions. This is generalized for larger cardinals by the following definition:
	
	\begin{mydef}
		Let $(\dP,\leq,\leq_0)$ be a Prikry-type forcing and $\delta,\mu$ be cardinals.
		\begin{enumerate}
			\item We say that $(\dP,\leq,\leq_0)$ has the \emph{direct $(\delta,\mu)$-covering property} if whenever $p\in\dP$ and $\tau$ is a $\dP$-name such that $p\Vdash\tau\in[\check{\mu}]^{<\check{\delta}}$, there is $q\leq_0p$ and $x\in[\mu]^{<\delta}$ such that $p\Vdash\tau\subseteq\check{x}$.
			\item We say that $(\dP,\leq,\leq_0)$ has \emph{direct $\mu$-decidability} if whenever $\tau$ is a $(\dP,\leq)$-name for an ordinal and $p\in\dP$ forces $\tau<\check{\mu}$, there is $\gamma<\mu$ and $p'\leq_0p$ such that $p'\Vdash\tau=\check{\gamma}$.
			\item We say that $(\dP,\leq,\leq_0)$ has \emph{almost direct $\mu$-decidability} if whenever $\tau$ is a $(\dP,\leq)$-name for an ordinal and $p\in\dP$ forces $\tau<\check{\mu}$, there is $\gamma<\mu$ and $p'\leq_0p$ such that $p'\Vdash\tau<\check{\gamma}$.
		\end{enumerate}
	\end{mydef}
	
	Clearly, direct $\mu$-decidability is equivalent to the direct $(2,\mu)$-covering property. Furthermore, whenever $\mu$ is regular, almost direct $\mu$-decidability is equivalent to the direct $(\mu,\mu)$-covering property.
	
	The proof of the iterability of the Prikry property (i.e. direct $2$-decidability; see e.g. \cite[Lemma 2.4]{JakobLevineFailureApproachability}) can easily be generalized to prove the following lemma:
	
	\begin{mylem}\label{PrikryIterEastonMu}
		Let $((\dP_{\alpha},\leq_{\alpha},\leq_{\alpha,0}),(\dot{\dQ}_{\alpha},\dot{\leq}_{\alpha},\dot{\leq}_{\alpha,0}))_{\alpha<\rho}$ be an Easton-support Magidor iteration of Prikry-type forcings and $\mu$ be a cardinal. Suppose that for all $\alpha<\rho$, $(\dP_{\alpha},\leq_{\alpha})$ preserves $\mu$ and forces that $(\dot{\dQ}_{\alpha},\dot{\leq}_{\alpha},\dot{\leq}_{\alpha,0})$ has direct $\mu$-decidability
		
		Then $(\dP_{\rho},\leq_{\rho},\leq_{\rho,0})$ has direct $\mu$-decidability.
	\end{mylem}
	
	One criterion for $\mu$-decidability is a sufficiently closed direct ordering (since we can always ``exclude'' ordinals using direct extensions). In the case where $\mu$ is not a strong limit, one even has a ``pumping-up'' phenomenon:
	
	\begin{mylem}[{\cite[Lemma 2.6]{JakobLevineFailureApproachability}}]\label{ClosureDecide}
		Let $(\dP,\leq,\leq_0)$ be a Prikry-type forcing and $\mu$ a regular cardinal such that $(\dP,\leq_0)$ is ${<}\,\mu$-closed. Then:
		\begin{enumerate}
			\item $\dP$ has direct $\gamma$-decidability for every $\gamma<\mu$.
			\item Whenever $\gamma,\gamma'<\mu$, $p\in\dP$ and $\tau$ is a $\dP$-name with $p\Vdash\tau\colon\check{\gamma}\to\check{\gamma}'$, there is $q\leq_0p$ and $f\colon\gamma\to\gamma'$ such that $q\Vdash\tau=\check{f}$.
			\item If $\mu$ is not a strong limit, $\dP$ has direct $\mu$-decidability.
		\end{enumerate}
	\end{mylem}
	
	Lastly, we have the following result on almost direct decidability for singular cardinals: Note that the corresponding lemma is false for direct decidability since it is possible to have forcings which add $\omega$-sequences to singular cardinals $\sup_n\kappa_n$ without adding any subsets to any $\kappa_n$.
	
	\begin{mylem}\label{SingAlmostDec}
		Let $\delta$ be a singular cardinal with $\cf(\delta)=\mu$. Let $(\dP,\leq,\leq_0)$ be a Prikry-type forcing that has almost direct $\mu$-decidability. Then $(\dP,\leq,\leq_0)$ has almost direct $\delta$-decidability.
	\end{mylem}
	
	\begin{proof}
		Let $\tau$ be a $\dP$-name for an ordinal and $p\in\dP$ which forces $\tau<\check{\delta}$. In $V$, let $f\colon\mu\to\delta$ be increasing and cofinal. Let $\sigma$ be a $\dP$-name for an ordinal such that $p$ forces $\tau<\check{f}(\sigma)$. Then by almost direct $\mu$-decidability there is $q\leq_0p$ which forces $\sigma<\check{\gamma}$ for some $\gamma<\mu$. Ergo $q$ forces $\tau<\check{f}(\check{\gamma})$.
	\end{proof}
	
	For full-support iterations, we have a strengthening of Lemma \ref{PrikryIterEastonMu} since only finitely many non-direct extensions are allowed in general:
	
	\begin{mylem}\label{DirectDecideCapture}
		Let $\rho$ be an ordinal. Assume that $((\dP_{\alpha},\leq_{\alpha},\leq_{\alpha,0}),(\dot{\dQ}_{\alpha},\dot{\leq}_{\alpha},\dot{\leq}_{\alpha,0}))$ is a full-support Magidor iteration of Prikry-type forcings of length $\rho$ and $\delta,\mu$ are cardinals. If for all $\alpha<\rho$, $\dP_{\alpha}$ preserves $\delta$ and $\mu$ and forces that $(\dot{\dQ}_{\alpha},\dot{\leq}_{\alpha},\dot{\leq}_{\alpha,0})$ has the $(\check{\delta},\check{\mu})$-covering property, then the full support limit $(\dP_{\rho},\leq_{\rho},\leq_{\rho,0})$ has the $(\delta,\mu)$-covering property.
	\end{mylem}
	
	\begin{proof}
		We prove the statement by induction on $\rho$. In case $\rho=1$, it is clear.
		
		Assume the statement holds for $\rho'$ and $\rho=\rho'+1$. Let $p\in\dP_{\rho}$ and $\tau$ be a $\dP_{\rho}$-name for a ${<}\,\delta$-sized subset of $\mu$ (forced by $p$). Then $p\uhr\rho'$ forces that $p(\rho')$ forces that $\tau$ is a ${<}\,\check{\delta}$-sized subset of $\check{\mu}$. Since $(\dot{\dQ}_{\rho'},\dot{\leq}_{\rho'},\dot{\leq}_{\rho',0})$ is forced to have the $(\check{\delta},\check{\mu})$-covering property, by the maximum principle, we can find a $\dP_{\rho'}$-name $\sigma$ and a $\dP_{\rho'}$-name $\dot{q}(\rho')$ such that
		$$p\uhr\rho'\Vdash\sigma\in[\check{\mu}]^{<\check{\delta}}\wedge \dot{q}(\rho')\Vdash\tau\subseteq\check{\sigma}.$$
		By the inductive hypothesis, there is a direct extension $q\uhr\rho'\leq_{\rho',0}p\uhr\rho'$ and $x\in[\mu]^{<\delta}$ such that $q\uhr\rho'\Vdash\sigma\subseteq\check{x}$. It follows that $(q\uhr\rho',q(\rho'))$ and $x$ are as required.
		
		Now assume that $\rho$ is a limit. Let $p$ be a condition in $\dP_{\rho}$ and $\tau$ a $\dP_{\rho}$-name for a ${<}\,\delta$-sized subset of $\mu$. Assume toward a contradiction that there is no direct extension of $p$ which covers $\tau$ by a ground-model set of size ${<}\,\delta$. By induction on $\alpha<\rho$ we define a condition $q\leq_0p$ such that
		$$q\uhr\alpha\Vdash_{\dP_{\alpha}}\neg\sigma_{\alpha}$$
		where (letting $(\dot{\dP}_{\alpha,\rho},\dot{\leq}_{\alpha,\rho},\dot{\leq}_{(\alpha,\rho),0})$ be a $\dP_{\alpha}$-name such that $\dP_{\rho}\cong\dP_{\alpha}*\dot{\dP}_{\alpha,\rho}$)
		$$\sigma_{\alpha}:=\exists r\in\dot{\dP}_{\alpha,\rho}\exists x\in[\check{\mu}]^{<\check{\delta}}(r\leq_{(\alpha,\rho),0}p\uhr[\alpha,\rho)\wedge r\Vdash\tau\subseteq\check{x}).$$
		Assume $q\uhr\alpha$ has been defined. We let $q(\alpha)$ be a $\dP_{\alpha}$-name for a condition such that $q\uhr\alpha$ forces that $q(\alpha)$ decides $\sigma_{\alpha+1}$ (this is possible as $\dot{\dQ}_{\alpha}$ is forced to have the Prikry property). We claim that $q\uhr\alpha\Vdash q(\alpha)\Vdash\neg\sigma_{\alpha+1}$. Otherwise there is $r\leq q\uhr\alpha$ forcing $q(\alpha)\Vdash\sigma_{\alpha+1}$ (since $q(\alpha)$ is forced to decide $\sigma_{\alpha+1}$). But then we can apply the maximum principle to find $\dot{r}\in\dot{\dP}_{\alpha+1,\rho}$ and $\dot{x}\in[\check{\mu}]^{<\check{\delta}}$ such that
		$$q\uhr\alpha\Vdash q(\alpha)\Vdash \dot{r}\Vdash\tau\subseteq\check{x}.$$
		But as in the case of the successor step, this already implies that there is $y\in[\mu]^{<\delta}$ such that $q\uhr\alpha$ forces that $(q(\alpha),\dot{r})$ forces $\tau\subseteq \check{y}$, a contradiction, as $q\uhr\alpha\Vdash\neg\sigma_{\alpha}$.
		
		Now assume that $\alpha\leq\rho$ is a limit and for every $\beta<\alpha$, $q\uhr\beta\Vdash\neg\sigma_{\beta}$. We claim that $q\uhr\alpha\Vdash\neg\sigma_{\alpha}$. Otherwise there is $r\leq q\uhr\alpha$ forcing $\sigma_{\alpha}$. However, then there is $\beta<\alpha$ such that $r\uhr\beta\Vdash r\uhr[\check{\beta},\check{\alpha})\leq_0 q\uhr[\check{\beta},\check{\alpha})$. This easily implies that $r\uhr\beta$ forces $\sigma_{\beta}$, a contradiction.
		
		Now assume $q$ has been defined. By the previous paragraph, we know that $q\Vdash\neg\sigma_{\rho}$. However, $\sigma_{\rho}$ is always true, since $q$ forces $\tau\subseteq\tau$, so we obtain a contradiction.
	\end{proof}
	
	In general, working with the direct extension ordering on a Magidor iteration is quite difficult since it does not quite behave like an iteration. Due to this, we will be using the following idea: Whenever $(\dP,\leq,\leq_0)$ is a Prikry-type poset and $(\dot{\dQ},\dot{\leq},\dot{\leq}_0)$ is a $\dP$-name for a poset, $(\dot{\dQ},\dot{\leq}_0)$ is a $(\dP,\leq)$-name for a poset.\footnote{The ordering $\leq_0$ on $\dP$ is just used to derive some regularity properties} As such, by Lemma \ref{Projection}, the identity function is a projection from $(\dP,\leq)\times\dT((\dP,\leq),(\dot{\dQ},\dot{\leq}_0))$ onto $(\dP,\leq)*(\dot{\dQ},\dot{\leq}_0)$. However, since $\leq$ refines $\leq_0$, we also obtain a projection from $(\dP,\leq_0)\times(\dT((\dP,\leq),(\dot{\dQ},\dot{\leq}_0)))$ onto the direct extension ordering on $\dP*\dot{\dQ}$:
	
	\begin{mylem}\label{TermspaceMagidor}
		Let $\kappa$ be an ordinal and let  $((\dP_{\alpha},\leq_{\alpha},\leq_{\alpha,0}),(\dot{\dQ}_{\alpha},\dot{\leq}_{\alpha},\dot{\leq}_{\alpha,0}))_{\alpha<\kappa}$ be a full support (Easton support) Magidor iteration of Prikry-type forcings.
		
		Let $\dT:=\prod_{\alpha<\kappa}\dT((\dP_{\alpha},\leq_{\alpha}),(\dot{\dQ}_{\alpha},\dot{\leq}_{\alpha,0}))$, where the product is taken with full support (Easton support). Then the identity is a projection from $\dT$ onto the direct extension ordering on the inverse (Easton) Magidor limit of the iteration $((\dP_{\alpha},\leq_{\alpha},\leq_{\alpha,0}),(\dot{\dQ}_{\alpha},\dot{\leq}_{\alpha},\dot{\leq}_{\alpha,0}))_{\alpha<\kappa}$.
	\end{mylem}
	
	\begin{proof}
		As in the proof of Lemma \ref{TermspaceIteration}, we do the proof just for the full support case.
		
		Let $\dP$ be the inverse Magidor limit of $((\dP_{\alpha},\leq_{\alpha},\leq_{\alpha,0}),(\dot{\dQ}_{\alpha},\dot{\leq}_{\alpha},\dot{\leq}_{\alpha,0}))$. Recall that $\dP$ consists of all functions $p$ on $\kappa$ such that $p\uhr\alpha\in\dP_{\alpha}$ for all $\alpha<\kappa$, directly ordered by $p\leq_0 q$ if and only if $p\uhr\alpha\leq_{\alpha,0}q\uhr\alpha$ for all $\alpha<\kappa$. Denote by $\leq_{\dT}$ the ordering on $\dT$.
		
		It is clear that the identity is a function from $\dT$ onto $\dP$ and preserves $\leq_0$: Let $t_0,t_1\in\dT$ with $t_0\leq_{\dT}t_1$. Then $t_0$ is a function on $\kappa$ such that for each $\alpha<\kappa$, $t_0(\alpha)$ is a $(\dP_{\alpha},\leq_{\alpha})$-name such that $1_{\dP_{\alpha}}\Vdash t_0(\alpha)\in\dot{\dQ}_{\alpha}$. It follows by induction on $\alpha$ that $t_0\uhr\alpha\in\dP_{\alpha}$ for every $\alpha<\kappa$ (since the iteration is taken with full support). For any $\alpha<\kappa$, $1_{\dP_{\alpha}}\Vdash t_0(\alpha)\leq_0t_1(\alpha)$. Ergo $t_0\uhr\alpha\Vdash_{(\dP_{\alpha},\leq_{\alpha})} t_0(\alpha)\dot{\leq}_{\alpha,0}t_1(\alpha)$. It follows again by induction that $t_0\uhr\alpha\leq_{\alpha,0}t_1\uhr\alpha$ for every $\alpha<\kappa$.
		
		Now we show that the identity is a projection. Let $p\in\dP$ and $t\leq_{\dP}p$. We want to find $p'\leq_{\dT}p$ with $p'\leq_{\dP}t$. By induction on $\alpha$, let $p'(\alpha)$ be a $(\dP_{\alpha},\leq_{\alpha})$-name for an element of $\dot{\dQ}_{\alpha}$ forced (over $(\dP_{\alpha},\leq_{\alpha})$) by $p'\uhr\alpha$ to be equal to $t(\alpha)$ and by conditions incompatible with $p'\uhr\alpha$ to be equal to $p(\alpha)$. Then clearly $p'\leq_{\dT}p$, since $p'(\alpha)$ is always forced to be below $p(\alpha)$ and $p'\leq_0t$: By induction, $p'\uhr\alpha\leq_0t\uhr\alpha$. Since $t\uhr\alpha\Vdash t(\alpha)\leq_0p(\alpha)$, $p'\uhr\alpha$ forces $p'(\alpha)=t(\alpha)\leq_0p(\alpha)$. It follows that $p'\leq_0p$.
	\end{proof}
	
	\section{The Laver-Ideal Property}
	
	In this section, we introduce the Laver-ideal property and prove a characterization for posets which force it. This expands on \cite[Section 3]{JakobLevineFailureApproachability}. 
	
	\begin{mydef}
		Let $\mu<\nu$ be regular cardinals. We let $\LIP(\mu,\nu,\kappa)$ state that there exists $I$, a ${<}\,\nu$-complete and normal ideal over $[\kappa]^{<\nu}$ such that there is a set $B\subseteq I^+$ which is dense in $I^+$ with respect to $\subseteq$ and strongly ${<}\,\mu$-directed closed.
	\end{mydef}
	
	Classically, by work of Laver (unpublished) and Galvin-Jech-Magidor (see \cite[Theorem 2]{GalvinJechMagidorIdealGame}), whenever $\mu$ is regular and $\nu>\mu$, the Levy collapse $\Coll(\mu,<\nu)$, which forces $\nu=\mu^+$, also forces $\LIP(\mu,\nu,\nu)$. In this section, we prove a characterization for posets which force $\LIP$ (which basically follows from their proof, combined with ideas of Foreman) and show that when starting from a supercompact cardinal instead of a measurable cardinal, we even obtain $\LIP(\mu,\nu,\kappa)$ for all $\kappa\geq\nu$. As will become apparent in the proof of the main theorem, we do this because we later have to show that the direct extension ordering on the tail of our iteration forces enough instances of $\LIP$ to make other parts of the iteration work.
	
	\begin{mydef}\label{LIPForcingDef}
		Let $\mu<\nu$ be regular cardinals such that $\nu$ is measurable. Let $\dP$ be a poset. We say that $\dP$ is a \emph{$\LIP(\mu,\nu)$-forcing} if the following holds:
		\begin{enumerate}
			\item $\dP$ is strongly ${<}\,\mu$-directed closed and $\nu$-cc.
			\item $\dP\subseteq V_{\nu}$.
			\item There is a club $C\subseteq\nu$ such that whenever $\alpha\in C$ is inaccessible, there is a projection $\pi\colon\dP\to(\dP\cap V_{\alpha})$ such that whenever $X\subseteq\dP$ is weakly directed, $|X|<\mu$ and $p\in\dP\cap V_{\alpha}$ satisfies $p\leq\pi(q)$ for every $q\in X$, there is $p^*\in\dP$ such that $\pi(p^*)=p$ and $p\leq q$ for every $q\in X$.
		\end{enumerate}
	\end{mydef}
	
	Condition (3) is a strengthening of the statement that $\dP/(\dP\cap V_{\alpha})$ is strongly ${<}\,\mu$-directed closed but it is in general easier to verify.
	
	\begin{mylem}\label{LIPForcingLemma}
		Let $\mu<\nu$ be regular cardinals such that $\nu$ is supercompact. Let $\dP$ be a $\LIP(\mu,\nu)$-forcing. Then whenever $\kappa\geq\nu$ is any cardinal, $\dP$ forces $\LIP(\mu,\nu,\kappa)$.
	\end{mylem}
	
	\begin{proof}
		Let $U$ be a normal ${<}\,\nu$-complete ultrafilter over $[\kappa]^{<\nu}$ and let $j\colon V\to M$ be the associated elementary embedding. Let $C$ be the club witnessing Definition \ref{LIPForcingDef} (3). Then $\nu\in j(C)$ is inaccessible and so there is a projection $j(\dP)\to(j(\dP)\cap V_{\nu}^M)$ and $j(\dP)\cap V_{\nu}^M=\dP$ (this follows from the fact that $V_{\nu}^M=V_{\nu}^V$ and $j\upharpoonright V_{\nu}=id_{V_{\nu}}$). It follows that whenever $G$ is $\dP$-generic and $H$ is any $j(\dP)/G$-generic filter over $V[G]$, $j$ lifts to $j^+\colon V[G]\to M[G][H]$ in $V[G][H]$ (since $j[G]=G$). We can now define our ideal witnessing $\LIP(\mu,\nu,\kappa)$ as follows: Given $X\subseteq[\kappa]^{<\nu}$ in $V[G]$, let $X\in I$ if and only if for every $j(\dP)/G$-generic filter $H$, $j[\kappa]\notin j^+(X)$.
		
		\begin{myclaim}
			$I$ is ${<}\,\nu$-complete and normal.
		\end{myclaim}
		
		\begin{proof}
			Let $(X_i)_{i<\gamma}$ ($\gamma<\nu$) be a sequence of elements of $I$. Let $H$ be any $j(\dP)/G$-generic filter over $V[G]$. Then $j^+((X_i)_{i<\gamma})=(j^+(X_i))_{i<\gamma}$ and so
			$$j[\kappa]\notin j^+\left(\bigcup_{i<\gamma}X_i\right)=\bigcup_{i<\gamma}j^+(X_i)$$
			i.e. $\bigcup_{i<\gamma}X_i\in I$.
			
			For normality, let $A\in I^+$ and let $f$ be a regressive function on $A$. Let $H$ be a $j(\dP)/G$-generic filter over $V[G]$ such that $j[\kappa]\in j^+(A)$. Since $f$ is a regressive function on $A$, $j(f)$ is a regressive function on $j(A)$ and so $j(f)(j[\kappa])\in j[\kappa]$. Let $j(f)(j[\kappa])=j(\alpha)$ for some $\alpha\in\kappa$. Then, clearly,
			$$j[\kappa]\in j^+\left(\{x\in A\;|\;f(x)=\alpha\}\right)$$
			and thus $f$ is constant on the set $\{x\in A\;|\;f(x)=\alpha\}$ which is a member of $I^+$.
		\end{proof}
		
		Now we have to construct the dense strongly ${<}\,\mu$-directed closed subset. This essentially follows from Foreman's duality theorem (see \cite[Theorem 7.14]{ForemanChapter}), which states that $\mathcal{P}([\kappa]^{<\nu})/I$ is forcing equivalent to the quotient $j(\dP)/G$, but we give an explicit argument here. For any $x\in V$, let $f_x\in V$ be a function representing $x$ in the ultrapower (i.e. $j(f_x)(j[\kappa])=x$). Also let $\dot{j}^+$ and $\dot{I}$ be names for $j^+$ and $I$ respectively (in the corresponding poset). By the definition of the lift, it follows that for any $\dP$-name $\tau$, the $\dP*j(\dP)/\Gamma_{\dP}$-name $\dot{j}^+(\tau)$ is equal to the $j(\dP)$-name $j(\tau)$.
		
		We first note the following: If $A,B\in I^+$ and $A\Vdash\check{B}\in\Gamma_{I^+}$, then $A\smallsetminus B\in I$: Otherwise $A\smallsetminus B$ would be a condition in $I^+$ that is below $A$ and incompatible with $B$. We write $A\subseteq^*B$ if $A\smallsetminus B\in I$.
		
		Let $X\in V[G]$ be a subset of $[\kappa]^{<\nu}$, $X=\tau^G$. Let $p\in G$ be such that $p\Vdash\tau\notin\dot{I}$. It follows that $p$ forces the existence of some $q\in j(\dP)/\Gamma_{\dP}$ which forces $j[\check{\kappa}]\in j(\tau)$. By the maximum principle, we can fix a name $\dot{q}$ such that $p\Vdash\dot{q}\Vdash j[\check{\kappa}]\in j(\tau)$.
		
		\begin{myclaim}
			For $U$-many $x\in[\kappa]^{<\nu}$, $\alpha:=x\cap\kappa\in C$, $p\in\dP\cap V_{\alpha}$ and $f_{\dot{q}}(x)$ is a $(\dP\cap V_{\alpha})$-name for a condition in $\dP/(\Gamma_{\dP\cap V_{\alpha}})$ such that $(p,f_{\dot{q}}(x))$ (as a condition in $\dP$) forces $\check{x}\in\tau$.
		\end{myclaim}
		
		\begin{proof}
			In $M$, $\nu=j[\kappa]\cap\nu\in j(C)$, $j(p)=p\in j(\dP)\cap V_{\nu}$ and $j(f_{\dot{q}})(j[\kappa])=\dot{q}$ is a $(j(\dP)\cap V_{\nu})$-name for a condition in $j(\dP)/(\Gamma_{j(\dP)\cap V_{\nu}})$ such that $j(p,f_{\dot{q}}(j[\kappa]))=(j(p),j(f_{\dot{q}})(j[\kappa]))=(p,\dot{q})$ forces $j[\check{\kappa}]\in j(\tau)$. This implies the claim, since $U=\{X\subseteq[\kappa]^{<\nu}\;|\;j[\kappa]\in j(X)\}$.
		\end{proof}
		
		So whenever $\tau$ is a $\dP$-name, $p\in\dP$ forces $\tau\notin\dot{I}$ and $\dot{q}$ is such that $p\Vdash_{\dP}\dot{q}\Vdash_{j(\dP)/\Gamma_{\dP}}j[\check{\kappa}]\in j(\tau)$, let $A(p,\dot{q},\tau)$ be the set of all $x\in[\kappa]^{<\nu}$ such that $\alpha:=x\cap\nu\in C$ and $(p,f_{\dot{q}}(\alpha))\Vdash\check{x}\in\tau$. When $G$ is $\dP$-generic over $V$ and $p\in G$, let $B(p,\dot{q},\tau,G)$ consist of all those $x\in[\kappa]^{<\nu}$ such that $(p,f_{\dot{q}}(x))\in G$. It follows that $B(p,\dot{q},\tau,G)\subseteq\tau^G$.
		
		\begin{myclaim}
			In $V[G]$, $B(p,\dot{q},\tau,G)\in I^+$.
		\end{myclaim}
		
		\begin{proof}
			Let $H$ be $j(\dP)/G$-generic over $V[G]$ with $\dot{q}^G\in H$. It follows that in $M[G*H]$, $j[\kappa]\in B(p,\dot{q},j(\tau),G*H)=j^+(B(p,\dot{q},\tau))$, since $(p,\dot{q}^G)\in G*H=j^+(G)$. Ergo there is a $j(\dP)/G$-generic filter $H$ such that $j[\kappa]\in j^+(B(p,\dot{q},\tau))$ so $B(p,\dot{q},\tau)\notin I$.
		\end{proof}
		
		So we can let $B$ be the collection of $B(p,\dot{q},\tau,G)$ whenever $p\in G$, $\tau$ is a $\dP$-name, $\dot{q}$ is a $\dP$-name for an element of $j(\dP)/\Gamma_{\dP}$ and $(p,\dot{q})$ forces $j[\check{\kappa}]\in j(\tau)$. It follows that $B$ is dense in $I^+$. We will show that $B$ is strongly ${<}\,\mu$-directed closed. To this end, let $D\subseteq B$ with size ${<}\,\mu$ be weakly directed, i.e. for any $B(p_0,\dot{q}_0,\tau_0,G),B(p_1,\dot{q}_1,\tau_1,G)\in D$, there is $B(p^*,\dot{q}^*,\tau^*,G)$ which is almost contained in $B(p_0,\dot{q}_0,\tau_0,G)$ and $B(p_1,\dot{q}_1,\tau_1,G)$. We will show that $\bigcap D$ is in $I^+$.
		
		Write $D=\{B(p_i,\dot{q}_i,\tau_i,G)\;|\;i<\gamma\}$. Since $p_i\in G$ for every $i<\gamma$ there is $p\in G$ such that $p\leq p_i$ for every $i<\gamma$ (since $\dP$ is in particular ${<}\,\mu$-distributive). Again by the distributivity, we can assume that the sequence $(p_i,\dot{q}_i,\tau_i)_{i<\gamma}$ is in $V$. Lastly, assume that $p$ forces that the collection $\{B(\check{p}_i,\check{\dot{q}}_i,\check{\tau}_i,\Gamma_{\dP})\;|\;i<\gamma\}$ is directed (note that we actually mean $\check{\tau}$ and $\check{\dot{q}_i}$, since we work with $\tau$ and $\dot{q}$ as $\dP$-names even in the forcing extension).
		
		\begin{myclaim}\label{LIPLemmaClaim3}
			In $V$, there are $U$-many $x\in[\kappa]^{<\nu}$ such that $x\cap\nu\in C$ is inaccessible and $p$ forces that $\{f_{\dot{q}_i}(x)\;|\;i<\gamma\}$ is a weakly directed subset of $\dP$.
		\end{myclaim}
		
		\begin{proof}
			Let $G'$ be an arbitrary $\dP$-generic filter containing $p$. Let $i<j<\gamma$ and let $k<\gamma$ be such that $B(p_k,\dot{q}_k,\tau_k,G')\subseteq^* B(p_i,\dot{q}_i,\tau_i,G'),B(p_j,\dot{q}_j,\tau_j,G')$. Let $H$ be $j(\dP)/G'$-generic over $V[G']$ containing $\dot{q}_k^{G'}$. It follows that $j[\kappa]\in j^+(B(p_k,\dot{q}_k,\tau_k,G'))$, since $(p,j(f_{\dot{q}_k})(j[\kappa]))\in G'*H$. This implies $j[\kappa]\in j^+(B(p_i,\dot{q}_i,\tau_i,G'))$ and $j[\kappa]\in j^+(B(p_i,\dot{q}_i,\tau_i,G'))$, i.e. $(p,j(f_{\dot{q}_i})(j[\kappa]))\in G'*H$,$(p,j(f_{\dot{q}_j})(j[\kappa]))\in G'*H$ (since $j[\kappa]\notin j^+(X)$ whenever $X\in I$). As $H$ was arbitrary, $j(f_{\dot{q}_k})(j[\kappa])$ forces $j(f_{\dot{q}_i})(j[\kappa])\in G'*\Gamma_{j(\dP)/G'}$ and $j(f_{\dot{q}_j})(j[\kappa])\in G'*\Gamma_{j(\dP)/G'}$. As $G'$ was arbitrary, $p$ forces that $j(f_{\dot{q}_k})(j[\kappa])$ forces $j(f_{\dot{q}_i})(j[\kappa])\in\Gamma_{j(\dP)}$ and $j(f_{\dot{q}_j})(\nu)\in\Gamma_{j(\dP)}$. This implies the claim by the definition of the ultrapower.
		\end{proof}
		
		Now we show that the quotient forcing is strongly directed-closed:
		
		\begin{myclaim}
			For any $\alpha\in C$ which is inaccessible, $\dP\cap V_{\alpha}$ forces that $\dP/(\Gamma_{\dP\cap V_{\alpha}})$ is strongly ${<}\,\mu$-directed closed.
		\end{myclaim}
		
		\begin{proof}
			Let $G_{\alpha}$ be $\dP\cap V_{\alpha}$-generic. Let $X\in V[G_{\alpha}]$ be a weakly directed subset of $\dP/G_{\alpha}$ of size ${<}\,\mu$. It follows that $X\in V$. It also follows that $\pi[X]\subseteq\dP\cap V_{\alpha}$ -- where $\pi\colon\dP\to(\dP\cap V_{\alpha})$ is the projection -- and so there is $p\in G_{\alpha}$ with $p\leq\pi(q)$ for every $q\in X$. By Definition \ref{LIPForcingDef} (3), there is $p^*\in\dP$ such that $\pi(p^*)=p$ and $p\leq q$ for every $q\in X$. Ergo, $p^*\in\dP/G_{\alpha}$ and $p^*$ is a lower bound of $X$.
		\end{proof}
		
		So let $X$ be the set witnessing Claim \ref{LIPLemmaClaim3}. For every $x\in X$, $p$ forces that there exists a lower bound $f_{\dot{q}}(x)$ of $\{f_{\dot{q}_i}(x)\;|\;i<\gamma\}$. We are done after showing:
		
		\begin{myclaim}
			In $V[G]$, $\bigcap_{i<\gamma}B(p_i,\dot{q}_i,\tau_i)\in I^+$.
		\end{myclaim}
		
		\begin{proof}
			In $V$, for $U$-many $x\in[\kappa]^{<\nu}$, $\alpha:=x\cap\nu\in C$ and $f_{\dot{q}}(x)$ is a $\dP\cap V_{\alpha}$-name for a condition in $\dP/(\Gamma_{\dP\cap V_{\alpha}})$ forced by $p$ to extend every $f_{\dot{q}_i}(x))$. Ergo $j(f_{\dot{q}})(j[\kappa])$ is a $\dP$-name for a condition in $j(\dP)/\Gamma$ forced by $p$ to extend every $j(f_{\dot{q}_i})(j[\kappa])$. So let $H$ be any $j(\dP)/G$-generic filter containing $j(f_{\dot{q}})(j[\kappa])$. Then $j(f_{\dot{q}_i})(j[\kappa])\in H$ for every $i<\gamma$ and moreover $(p_i,j(f_{\dot{q}_i})(j[\kappa]))\in G*H$ for every $i<\gamma$ (since $p\leq p_i$). It follows that
			$$j[\kappa]\in j^+\left(\bigcap_{i<\gamma}B(p_i,\dot{q}_i,\tau_i)\right)=\bigcap_{i<\gamma}j^+(B(p_i,\dot{q}_i,\tau_i)).$$
			Ergo $\bigcap_{i<\gamma}B(p_i,\dot{q}_i,\tau_i)\notin I$.
		\end{proof}
		
		This finishes the proof of Lemma \ref{LIPForcingLemma}.
		
	\end{proof}
	
	\section{Collapsing without Approachability}\label{CollapseNoApp}
	
	As stated in the discussion below Lemma \ref{InducedColoring}, the main obstacle to obtaining a model where there are many non-approachable points in the successor of a singular cardinal is finding a poset which collapses the successor of a singular cardinal of a given cofinality without making it approachable with respect to any normal, subadditive coloring. To achieve this for $\kappa=\mu^+$, $\mu$ singular, Shelah used $\Coll(\cf(\mu),\mu)$. Assuming the $\GCH$, this poset has size $\mu$ and thus preserves $\mu^+$. In particular, it does not make $\mu^+$ approachable with respect to any normal, subadditive coloring. However, this approach has two major drawbacks: On one hand, since larger collapses would also collapse $\kappa$, it is specific to obtaining non-approachable points of cofinality $\cf(\mu)^+$ and, since $\kappa$ always needs to be preserved, it cannot be iterated. As such, it was a major open problem whether there could be a poset which can collapse $\kappa$ as a cardinal without making it approachable with respect to any normal subadditive coloring.	
	
	A positive solution to this problem was found by the author and Levine for successors of singulars of countable cofinality (see \cite{JakobLevineFailureApproachability}). In the cited work, we relied on having $\LIP(\mu,\aleph_n,\aleph_n)$ for infinitely many $n\in\omega$. In this section we will show that it is possible to obtain a poset which accomplishes the same task but can be defined from weaker hypotheses -- assuming $\LIP(\mu,\aleph_k,\aleph_n)$ for a single $k$ and infinitely many $n\in\omega$. This is crucial in making the later iteration work since we will just need a single generically supercompact cardinal instead of infinitely many generically measurable cardinals which is much easier to obtain in practice.
	
	We first introduce our new variant of Namba forcing. Just like in the case of supercompact Prikry forcing, our trees will consist of sequences of elements of some $[\lambda]^{<\kappa}$ as opposed to sequences of ordinals. However, unlike supercompact Prikry forcing, it can be defined for successor cardinals since we are using $\LIP$-ideals as opposed to supercompact measures.
	
	\begin{mydef}
		Let $(\kappa_n)_{n\in\omega}$ be an increasing sequence of regular cardinals and $\mu<\kappa_0$ regular. Assume that for each $n\in\omega$, there exists a ${<}\,\mu^+$-complete ideal $I_n$ over $[\kappa_n]^{<\mu^+}$ such that there is a set $B_n\subseteq I_n^+$ which is dense in $(I_n^+,\subseteq)$ and strongly ${<}\,\mu$-directed closed. Let $\dL(\mu,(\kappa_n,I_n,B_n)_{n\in\omega})$ consist of all $p\subseteq\bigcup_{n\in\omega}\prod_{k<n}[\kappa_k]^{<\mu^+}$ such that
		\begin{enumerate}
			\item $p$ is closed under restriction.
			\item There is some $\stem(p)\in p$ such that for all $s\in p$, $\stem(p)\subseteq s$ or $s\subseteq\stem(p)$ and whenever $\stem(p)\subseteq s$,
			$$\osucc_p(s):=\{x\in[\kappa_{|s|}]^{<\mu^+}\;|\;s^{\frown}x\in p\}\in B_{|s|}.$$
		\end{enumerate}
		We let $p\leq q$ if $p\subseteq q$. We let $p\leq_0q$ if $p\subseteq q$ and $\stem(p)=\stem(q)$.
	\end{mydef}
	
	Since each $B_n$ is dense in $I_n^+$, we can refine trees which split into elements of $I_n^+$ to conditions in $\dL(\mu,(\kappa_n,I_n,B_n)_{n\in\omega})$:
	
	\begin{mybem}
		Let $p\subseteq\bigcup_{n\in\omega}\prod_{k<n}[\kappa_k]^{<\mu^+}$ be closed under restriction such that there is some $\stem(p)\in p$ such that for all $s\in p$, $\stem(p)\subseteq s$ or $s\subseteq\stem(p)$ and whenever $\stem(p)\subseteq s$,
		$$\osucc_p(s)\in I_{|s|}^+.$$
		For each $s\in p$ with $\stem(p)\subseteq s$, choose $A_s\subseteq\osucc_p(s)$ with $A_s\in B_{|s|}$. Let $q$ consist of all $t\in p$ such that whenever $s$ is an initial segment of $t$ with $\stem(p)\subseteq s$, $t(|s|)\in A_s$. It follows that, for each $t\in q$, $\osucc_q(t)=A_t$. Therefore, $q\subseteq p$ and $q\in\dL(\mu,(\kappa_n,I_n,B_n)_{n\in\omega})$.
	\end{mybem}
	
	Ergo, in many constructions, we will only aim to produce a tree with $(I_n)_{n\in\omega}$-positive splitting sets, knowing that any such tree can be refined into a condition in $\dL(\mu,(\kappa_n,I_n,B_n)_{n\in\omega})$.
	
	The forcing functions similarly to regular Namba forcing defined using a sequence $(I_n)_{n\in\omega}$ of ideals on $(\kappa_n)_{n\in\omega}$. The main advantage (which is crucial for this paper) is that we are only assuming a single instance of $\LIP$. Despite this, our Namba forcing has all of the necessary regularity properties, most importantly the strong Prikry property (see Lemma \ref{StrongPrikry}), the ability to decide names for small ordinals using direct extensions (see Lemma \ref{LDirectDecide}) and a sufficiently strongly directed-closed direct extension ordering (see \ref{LClosure}). For the first two results, we will follow Cox-Krueger (see \cite{CoxKruegerNamba}). Previously, similar results were obtained by Cummings and Magidor (see \cite{CumMagMMWeakSquare}).
	
	\begin{myass}
		$(\kappa_n)_{n\in\omega}$ is an increasing sequence of regular cardinals and $\mu<\kappa_0$ is regular. For each $n\in\omega$, $I_n$ is a ${<}\,\mu^+$-complete ideal over $\kappa_n$ and there is a set $B_n\subseteq I_n^+$ which is dense in $I_n^+$ and strongly ${<}\,\mu$-directed closed. Define $\kappa:=\sup_n\kappa_n$ and $\kappa^*:=(\sup_n\kappa_n)^+$.
	\end{myass}
	
	We also fix the following pieces of notation:
	
	\begin{mydef}
		Let $p\in\dL(\mu,(\kappa_n,I_n,B_n)_{n\in\omega})$ and let $t\in p$.
		\begin{enumerate}
			\item We define $p\uhr t:=\{s\in p\;|\;t\subseteq s\vee s\subseteq t\}$.
			\item For $n\in\omega$, we let
			$$\osucc_p^n(t):=\{r\in\prod_{i\in[|t|,|t|+n-1]}[\kappa_i]^{<\mu^+}\;|\;t^{\frown}r\in p\}.$$
		\end{enumerate}
	\end{mydef}
	
	Clearly, $p\uhr t\in\dL(\mu,(\kappa_n,I_n,B_n)_{n\in\omega})$ whenever $p\in\dL(\mu,(\kappa_n,I_n,B_n)_{n\in\omega})$ and $t\in p$. If $\stem(p)\subseteq t$, then $\stem(p\uhr t)=t$.
	
	We first prove that $(\dL(\mu,(\kappa_n,I_n,B_n)_{n\in\omega}),\leq,\leq_0)$ is a Prikry-type forcing and derive its regularity properties from a sufficiently closed direct extension ordering.
	
	We start by proving that $\dL(\mu,(\kappa_n,I_n,B_n)_{n\in\omega})$ has the \emph{strong Prikry property}. To do so we introduce the notion of a \emph{fusion sequence}: Given $p,q\in\dL(\mu,(\kappa_n,I_n,B_n)_{n\in\omega})$ and $n\in\omega$, we let $p\leq_nq$ if and only if $p\leq_0 q$ and for all $k<n$, $\osucc_p^k(\stem(p))=\osucc_q^k(\stem(p))$ (recall that $\stem(q)=\stem(p)$). A sequence $(p_n)_{n\in\omega}$ is a \emph{fusion sequence} if $p_{n+1}\leq_np_n$ for all $n\in\omega$. Whenever $(p_n)_{n\in\omega}$ is a fusion sequence, the intersection $\bigcap_{n\in\omega}p_n$ (sometimes referred to as the \emph{fusion limit}) is a condition in $\dL(\mu,(\kappa_n,I_n,B_n)_{n\in\omega})$ and extends every $p_n$.
	
	\begin{mylem}\label{StrongPrikry}
		Let $p\in\dL(\mu,(\kappa_n,I_n,B_n)_{n\in\omega})$ and let $D\subseteq\dL(\mu,(\kappa_n,I_n,B_n)_{n\in\omega})$ be open and dense below $p$. Then there is $q\leq_0p$ and $n\in\omega$ such that whenever $s\in q$, $|s|=n$, $q\uhr s\in D$.
	\end{mylem}
	
	\begin{proof}
		We will call a condition $q\leq p$ \emph{bad} if there is no $r\leq_0q$ and $n\in\omega$ such that $r\uhr s\in D$ whenever $s\in r$ with $|s|=n$. We will assume toward a contradiction that the lemma fails, i.e. $p$ itself is bad.
		
		\setcounter{myclaim}{0}
		
		\begin{myclaim}\label{StrongPrikryClaim1}
			Whenever $q\leq p$ is bad, the set of all $x\in\osucc_q(\stem(q))$ such that $q\uhr(\stem(q)^{\frown}x)$ is not bad is in $I_{|\stem(q)|}$.
		\end{myclaim}
		
		\begin{proof}
			Assume the claim fails, i.e. the set of all $x\in\osucc_q(\stem(q))$ such that $q\uhr(\stem(q)^{\frown}x)$ is not Then by the ${<}\,\mu^+$-completeness of $I_{|\stem(q)|}$ there is a set $A\subseteq\osucc_q(\stem(q))$ in $I_{|\stem(q)|}^+$ and $n\in\omega$ such that for every $x\in A$ there is $r_x\leq_0q\uhr(\stem(q)^{\frown}x)$ such that whenever $s\in r_x$, $|s|=n$, $r_x\uhr s\in D$. However, then we can let
			$$r:=\bigcup\{r_x\;|\;x\in A\}.$$
			It follows that $r\leq_0q$ and whenever $s\in r$, $|s|=n$, $r\uhr s=r_{s(|\stem(q)|)}\uhr s\in D$, so $q$ is not bad, witnessed by $r$ and $n$, a contradiction.
		\end{proof}
		
		Now we build a fusion sequence $(p_n)_{n\in\omega}$ with $p_0\leq p$ such that whenever $s\in p_n$, $|s|=\stem(p)+n$, $p_n\uhr s$ is bad. Let $p_0:=p$ which obviously works by assumption. Given $p_n$ and $s\in p_n$ with $|s|=\stem(p)+n$, we know by the inductive hypothesis that $p_n\uhr s$ is bad. By Claim \ref{StrongPrikryClaim1}, the set $A$ of all $x\in\osucc_{p_n\uhr s}(s)$ such that $p_n\uhr(s^{\frown}x)$ is bad is in $I_{|\stem(q)|}^+$. So we can let $p_n^s$ consist of all $t\in p_n\uhr s$ such that $t(|s|)\in A$. It follows that $p_n^s\leq_0p_n\uhr s$ and we can let $p_{n+1}:=\bigcup_{s\in p_n,|s|=n}p_n^s$. Then $p_{n+1}\leq_np_n$ and it satisfies the inductive hypothesis.
		
		Now let $p$ be the fusion limit of $(p_n)_{n\in\omega}$. Then whenever $s\in p$, $|s|\geq|\stem(p)|$, $p\uhr s$ is bad, since $p\uhr s\leq_0 p_n\uhr s$. However, since $D$ is open dense, there is $q\leq p$ with $q\in D$. Then $q\leq_0p\uhr(\stem(q))$ and so $p\uhr(\stem(q))$ is not bad, a contradiction.
	\end{proof}
	
	From this we can show that $\dL(\mu,(\kappa_n,I_n,B_n)_{n\in\omega})$ has the Prikry property (and in fact even direct $\mu$-decidability):
	
	\begin{mylem}\label{LDirectDecide}
		Let $\tau$ be an $\dL(\mu,(\kappa_n,I_n,B_n)_{n\in\omega})$-name and $p\in\dL(\mu,(\kappa_n,I_n,B_n)_{n\in\omega})$ such that $p\Vdash\tau<\check{\mu}$. Then there is $q\leq_0p$ and $\gamma<\mu$ such that $q\Vdash\tau=\check{\gamma}$.
	\end{mylem}
	
	\begin{proof}
		Let $D$ be the open dense set of conditions deciding $\tau$. By Lemma \ref{StrongPrikry}, there is $q\leq_0p$ and $n\in\omega$ such that whenever $s\in q$, $|s|=n$, $q\uhr s\in D$. Assume $q$ was chosen so that $n$ is minimal. We claim that $n=|\stem(p)|$, which implies the lemma.
		
		Assume toward a contradiction that $n>|\stem(p)|$. We will find $r\leq_0q$ such that whenever $s\in r$, $|s|=n-1$, $r\uhr s\in D$, thus contradicting the minimality of $n$. Let $s\in q$, $|s|=n-1$. Since $|s|\geq|\stem(p)|$, $\osucc_q(s)\in I_{|s|}^+$. For any $x\in\osucc_q(s)$, there is $\gamma_x$ such that $q\uhr(s^{\frown}x)\Vdash\tau=\check{\gamma}_x$. By the ${<}\,\mu^+$-completeness of $I_{|s|}$, there is $B_s\subseteq\osucc_q(s)$ and $\gamma_s$ such that for any $\alpha\in B_s$, $q\uhr(s^{\frown}x)\Vdash\tau=\check{\gamma}_s$. Now let $r^s:=\{t\in q\uhr s\;|\;|t|\leq|s|\vee t(|s|)\in B_s\}$. It follows that $r^s\leq_0 q\uhr(s^{\frown}\alpha)$ and $r^s\Vdash\tau=\check{\gamma}_s$, i.e. $r^s\in D$. Let $r:=\bigcup_{s\in q,|s|=n-1}r^s$. Then $r\leq_0q$ and whenever $s\in r$, $|s|=n-1$, $r\uhr s=r^s$ and thus $r\uhr s\in D$. This $r$ is what we wanted to construct.
	\end{proof}
	
	As is common with Prikry-type forcings, $\dL(\mu,(\kappa_n,I_n,B_n)_{n\in\omega})$ derives its regularity properties from the Prikry property together with a sufficiently closed direct extension ordering. For technical reasons -- most importantly, to show that our later iteration forces instances of $\LIP$ -- we need that the direct extension ordering is strongly ${<}\,\mu$-directed closed. To show this, we first prove some elementary statements regarding ideal combinatorics.
	
	\begin{mydef}
		Let $I$ and $J$ be ideals on $X$ and $Y$ respectively. We define the \emph{Fubini product} $I\otimes J$ (which is an ideal on $X\times Y$) by
		$$X\times Y\supseteq A\in I\otimes J:\Leftrightarrow\{x\in X\;|\;\{y\in Y\;|\;(x,y)\in A\}\in J^+\}\in I.$$
	\end{mydef}
	
	I.e., $A\in (I\otimes J)^+$ if and only if $\{x\in X\;|\;\{y\in Y\;|\;(x,y)\in A\}\in J^+\}\in I^+$.
	
	\begin{mylem}\label{FubiniAssociative}
		The Fubini product is associative. In other words, whenever $I,J,K$ are ideals (on $X,Y,Z$ respectively), then $(I\otimes J)\otimes K=I\otimes(J\otimes K)$.
	\end{mylem}
	
	\begin{proof}
		Let $A\subseteq X\times Y\times Z$. For $x\in X$, define $A_x:=\{(y,z)\in Y\times Z\;|\;(x,y,z)\in A\}$.
		\begin{align*}
			A & \in (I\otimes J)\otimes K \\
			& \Leftrightarrow \{(x,y)\in X\times Y\;|\;\{z\in Z\;|\;(x,y,z)\in A\}\in K^+\}\in (I\otimes J) \\
			& \Leftrightarrow \{x\in X\;|\;\{y\in Y\;|\;\{z\in Z\;|\;(x,y,z)\in A\}\in K^+\}\in J^+\}\in I \\
			& \Leftrightarrow \{x\in X\;|\;\{y\in Y\;|\;\{z\in Z\;|\;(y,z)\in A_x\}\in K^+\}\in J^+\}\in I \\
			& \Leftrightarrow \{x\in X\;|\;A_x\in(J\otimes K)^+\}\in I \\
			& \Leftrightarrow \{x\in X\;|\;\{(y,z)\in Y\times Z\;|\;(x,y,z)\in A\}\in(J\otimes K)^+\}\in I \\
			& \Leftrightarrow A\in I\otimes(J\times K)
		\end{align*}
		This finishes the proof.

	\end{proof}
	
	In particular, we may now write expressions such as ``$I_0\otimes\dots\otimes I_n$'' without using brackets to indicate the order of operations, since e.g. $I_0\otimes(I_1\otimes (I_2\otimes I_3))=(I_0\otimes I_1)\otimes (I_2\otimes I_3)$.
	
	The Fubini product is related to Namba forcing as follows: Given finitely many sets $(X_k)_{k<n}$ and ideals $(J_k)_{k<n}$, we say that $p\subseteq\bigcup_{l<n}\prod_{k\leq l}X_k$ is a \emph{$(J_k)_{k<n}$-Laver-tree} if $p$ is nonempty, closed under restriction and for every $s\in p$ with $|s|<n$, $\osucc_p(s):=\{x\in X_{|s|}\;|\;s^{\frown}x\in p\}\in J_{|s|}^+$. We let $T(p):=\{s\in p\;|\;|s|=n\}$.
	
	\begin{mylem}\label{ProductTree}
		Let $J_0,\dots,J_{n-1}$ be ideals. Then a set $A\subseteq X_0\times\dots\times X_{n-1}$ is in $(J_0\otimes\dots\otimes J_{n-1})^+$ if and only if there is a $(J_k)_{k<n}$-Laver-tree $p$ such that $T(p)\subseteq A$.
	\end{mylem}
	
	\begin{proof}
		We prove both directions simultaneously by induction with the beginning being clear.
		
		Assume the statement has been shown for $J_0,\dots,J_{n-1}$ and let $J_n$ be some ideal.
		
		Let $A\in(J_0\otimes\dots\otimes J_n)^+$. It follows that for $(J_0\otimes\dots\otimes J_{n-1})^+$ many $s\in\prod_{k<n}X_k$, $\{x\in X_n\;|\;s^{\frown}x\in A\}\in J_n^+$. By the inductive hypothesis, let $p$ be a $(J_k)_{k<n}$-Laver-tree such that whenever $s\in p$, $\{x\in X_n\;|\;s^{\frown}x\in A\}\in J_n^+$. Let $p'$ be the tree of all $n+1$-element sequences $s$ with $s\uhr n\in p$ and $(s\uhr n,s(n))\in A$. Then $p'$ is a $(J_k)_{k\leq n}$-Laver-tree and $T(p')\subseteq A$.
		
		On the other hand, suppose that $A\subseteq X_0\times\dots\times X_n$ and $p$ is a $(J_k)_{k\leq n}$-Laver tree such that $T(p)\subseteq A$. Whenever $s\in p$, $|s|=n$, then $\osucc_p(s)\in J_n^+$ by assumption. On the other hand, $p':=\{s\in p\;|\;|s|\leq n\}$ clearly is a $(J_k)_{k<n}$-Laver tree and $T(p')=\{s\in p\;|\;|s|=n\}$. Therefore,
		\begin{align*}
			\{(x_0,\dots,x_{n-1})\in X_0\times\dots\times X_{n-1}\;|\;\{x_n\in X_n\;|\;(x_0,\dots,x_n)\in A\}\in J_n^+\}
			\supseteq T(p')
		\end{align*}
		Since $T(p')\in (J_0\otimes\dots\otimes J_{n-1})^+$ by induction, $A\in (J_0\otimes\dots\otimes J_n)^+$.
	\end{proof}
	
	We also have to show that the completeness of ideals is preserved by Fubini products.
	
	\begin{mylem}\label{ProductComplete}
		Let $I$ and $J$ be ${<}\,\kappa$-complete. Then $I\otimes J$ is ${<}\,\kappa$-complete.
	\end{mylem}
	
	\begin{proof}
		Let $\mathcal{A}\subseteq I\otimes J$ be of size ${<}\,\kappa$. We want to show that $A:=\bigcup \mathcal{A}\in I\otimes J$. For every $x\in X$, $\{y\in Y\;|\;(x,y)\in A\}=\bigcup_{B\in\mathcal{A}}\{y\in Y\;|\;(x,y)\in B\}$. This set is in $J^+$ if and only if $\{y\in Y\;|\;(x,y)\in B\}\in J^+$ for some $B\in\mathcal{A}$, by the completeness of $J$. So
		$$\{x\in X\;|\;\{y\in Y\;|\;(x,y)\in A\}\in J^+\}=\bigcup_{B\in\mathcal{A}}\{x\in X\;|\;\{y\in Y\;|\;(x,y)\in B\}\in J^+\}$$
		which is in $I$ by completeness. Ergo, $A\in I\otimes J$.
	\end{proof}
	
	Now we can show the result we are actually after:
	
	\begin{mylem}\label{LClosure}
		The poset $(\dL(\mu,(\kappa_n,I_n,B_n)_{n\in\omega}),\leq_0)$ is strongly ${<}\,\mu$-directed closed.
	\end{mylem}
	
	\begin{proof}
		Let $X\subseteq\dL(\mu,(\kappa_n,I_n,B_n)_{n\in\omega})$ be weakly directed with respect to $\leq_0$. Let $p:=\bigcap X$. Note that $p$ is not necessarily a condition in $\dL(\mu,(\kappa_n,I_n,B_n)_{n\in\omega})$, since $X$ is only weakly directed. E.g. it is possible that $X$ contains two conditions $q_0$ and $q_1$ such that for some $\alpha\in\osucc_{q_0}(\stem(q_0))\cap\osucc_{q_1}(\stem(q_1))$, the $\osucc$-sets of $\stem(q_i)^{\frown}\alpha$ are disjoint for $i=0,1$ (as long as this does not occur for too many $\alpha$). In that case, the intersection $q_0\cap q_1$ contains a terminal node, namely $\stem(q_0)^{\frown}\alpha$, and is therefore not a condition in $\dL$. However, we will show that $p$ can be refined to a condition in $\dL(\mu,(\kappa_n,I_n,B_n)_{n\in\omega})$ with the same stem. Let $s:=\stem(p)$ (which is also the stem of every $q\in X$). We will assume for notational simplicity that $s=\emptyset$.
		
		\setcounter{myclaim}{0}
		
		We first prove a general claim. Let $\dL:=\dL(\mu,(\kappa_n,I_n,B_n)_{n\in\omega})$.
		
		\begin{myclaim}\label{LClosureClaim1}
			If $q_0,q_1\in\dL$, $\stem(q_0)=\stem(q_1)=\emptyset$ and $q_0\Vdash\check{q}_1\in\Gamma_{\dL}$, then for any $n\in\omega$,
			$$\osucc_{q_0}^n(\emptyset)\subseteq^*\osucc_{q_1}^n(\emptyset)$$
			where $\subseteq^*$ is defined with regard to $I_0\otimes\dots\otimes I_{n-1}$.
		\end{myclaim}
		
		\begin{proof}
			Assume toward a contradiction that $A:=\osucc_{q_0}^n(\emptyset)\smallsetminus\osucc_{q_1}^n(\emptyset)$ is $ (I_0\otimes\dots\otimes I_{n-1})$-positive. By Lemma \ref{ProductTree} there is an $(I_k)_{k<n}$-tree $T$ such that $[T]\subseteq A$. Let $q':=\{s\in q_0\;|\;s\uhr n\in T\}$. We claim that $q'$ is a condition in $\dL$: Let $s\in q'$. If $|s|\geq n$, then $\osucc_{q'}(s)=\osucc_{q_0}(s)\in I_{|s|}^+$ so we are done. If $|s|<n$, then $\osucc_{q'}(s)=\osucc_T(s)$, since $[T]\subseteq A\subseteq\osucc_{q_0}^n(\emptyset)$. But that set is in $I_{|s|}^+$ as well by assumption. It follows that $q'\leq_0q_0$ and $q'\perp q_1$, since for every $s\in q'$ with $|s|=n$, $s\in A$ and therefore $s\notin\osucc_{q_1}^n(\emptyset)$. This  contradicts the fact that $q_0\Vdash\check{q}_1\in\Gamma_{\dL}$.
		\end{proof}
		
		This claim implies that the first $\osucc$-set is as required:
		
		\begin{myclaim}\label{LClosureClaim2}
			$\osucc_p(\emptyset)\in I_0^+$.
		\end{myclaim}
		
		\begin{proof}
			We claim that the collection $\{\osucc_q(\emptyset)\;|\;q\in X\}$ is a weakly directed subset of $(B_0,\subseteq)$. Given $q_0,q_1\in X$, let $q_2\in X$ be such that $q_2\Vdash q_0,q_1\in\Gamma_{\dL}$. It follows from Claim \ref{LClosureClaim1} with $n=0$ that $\osucc_{q_2}(\emptyset)\subseteq^*\osucc_{q_0}(\emptyset)\cap\osucc_{q_1}(\emptyset)$, so $\osucc_{q_2}(\emptyset)$ witnesses that $\osucc_{q_0}(\emptyset)$ and $\osucc_{q_1}(\emptyset)$ are weakly compatible.
			
			Since $(B_0,\subseteq)$ is strongly ${<}\,\mu$-directed closed and since $\osucc_p(\emptyset)=\bigcap_{q\in X}\osucc_q(\emptyset)$, the latter set is in $I_0^+$.
		\end{proof}
		
		We also have that at least for ``most'' $t\in\osucc_p^n(\emptyset)$, the $\osucc$-sets are large enough:
		
		\begin{myclaim}\label{LClosureClaim3}
			For $n\in\omega$, $\{t\in\osucc_p^n(\emptyset)\;|\;\osucc_p(t)\in I_n\}\in I_0\otimes\dots\otimes I_{n-1}$.
		\end{myclaim}
		
		\begin{proof}
			Let $q_0,q_1\in X$. Then there is $q_2\in X$ such that $q_2\Vdash q_0,q_1\in\Gamma_{\dL}$. By Claim \ref{LClosureClaim1} we have that $\osucc_{q_2}^{k}(\emptyset)\subseteq^*\osucc_{q_0}^{k}(\emptyset),\osucc_{q_1}^{k}(\emptyset)$. Let $i<2$ be arbitrary. Then $\osucc_{q_2}^{n+1}\smallsetminus\osucc_{q_i}^{n+1}\in I_0\otimes\dots\otimes I_{n}$ and therefore
			$$\{t\in\osucc_{q_2}^n(\emptyset)\;|\;\{x\in[\kappa_{n-1}]^{<\mu^+}\;|\;(t,x)\in\osucc_{q_2}^{n+1}(\emptyset)\smallsetminus\osucc_{q_i}^{n+1}(\emptyset)\}\in I_n^+\}$$
			is in $I_0\otimes\dots\otimes I_{n-1}$
			Ergo, we have
			$$\{t\in\osucc_{q_2}^n(\emptyset)\;|\;\osucc_{q_2}(t)\not\subseteq^*\osucc_{q_i}(t)\}\in I_0\otimes\dots\otimes I_{n-1}.$$
			It follows that
			$$A_{q_0,q_1,q_2}:=\{t\in\osucc_{q_2}^n(s)\;|\;\exists i<2(t\notin\osucc_{q_i}(\emptyset)\vee\osucc_{q_2}(t)\not\subseteq^*\osucc_{q_i}(t))\}$$
			is in $I_0\otimes\dots\otimes I_{n-1}$ since it is the union of four sets in $I_0\otimes\dots\otimes I_{n-1}$. Let $A$ be the union of all $A_{q_0,q_1,q_2}$ where $q_0,q_1\in X$ and $q_2$ witnesses their weak compatibility. Then $A\in I_0\otimes\dots\otimes I_{n-1}$ since that ideal is ${<}\,\mu$-complete by Lemma \ref{ProductComplete}. It follows that whenever $t\in\osucc_p^n(\emptyset)\smallsetminus A$, the collection $\{\osucc_q(t)\;|\;q\in X\}$ is weakly directed (first note that $t\in q$ whenever $q\in X$): Given $q_0,q_1\in X$, let $q_2$ witness their weak compatibility. By assumption $t\notin A_{q_0,q_1,q_2}$ and so $\osucc_{q_2}(t)\subseteq^*\osucc_{q_0}(t),\osucc_{q_1}(t)$. So whenever $t\in\osucc_p^n(\emptyset)\smallsetminus A$, we have that $\osucc_p(t)=\bigcap_{q\in X}\osucc_q(t)\in I_{|t|}^+$, since $(B_n,\subseteq)$ is strongly ${<}\,\mu$-directed closed. It follows that
			$$\{t\in\osucc_p^n(\emptyset)\;|\;\osucc_p(t)\in I_n\}\subseteq A\in I_0\otimes\dots\otimes I_{n-1}.$$
		\end{proof}
		
		Now we let $q'$ consist of all $t\in p$ such that $\osucc_p(t)\in I_{|t|}^+$ and for every $n$,
		$$\{r\in\osucc_p^n(t)\;|\;\osucc_p^n(t^{\frown}r)\in I_{|t|+n}\}\in I_{|t|}\otimes\dots\otimes I_{|t|+(n-1)}.$$
		Let $q$ consist of all those $t\in q'$ such that every initial segment of $t$ is also in $q'$. We claim that $q\in\dL$ and $\stem(q)=\emptyset$. By Claim \ref{LClosureClaim2} and Claim \ref{LClosureClaim3}, $\emptyset\in q$ and so $q$ is nonempty. It is clearly closed under initial segments by construction. So we have to show that it splits widely enough. Let $t\in q$. By assumption, $\osucc_p(t)\in I_{|t|}^+$. Additionally, $\osucc_q(t)$ consists of all $x\in\osucc_p(t)$ such that $\osucc_p(t^{\frown}x)\in I_{|t|+1}^+$ and for every $n$,
		$$\{r\in\osucc_p^n(t^{\frown}x)\;|\;\osucc_p(t^{\frown}x^{\frown}r)\in I_{|t|+1+n}\}\in I_{|t|+1}\otimes\dots\otimes I_{|t|+n}.$$
		So suppose $\osucc_q(t)\in I_{|t|}$. Then $\osucc_p(t)\smallsetminus\osucc_q(t)\in I_{|t|}^+$ and so for every $x\in\osucc_p(t)\smallsetminus\osucc_q(t)$, either $\osucc_p(t^{\frown}x)\in I_{|t|+1}$ or there is some $n$ such that
		$$\{r\in\osucc_p^n(t^{\frown}x)\;|\;\osucc_p(t^{\frown}x^{\frown}r)\in I_{|t|+1+n}\}\in (I_{|t|+1}\otimes\dots\otimes I_{|t|+n})^+.$$
		By the ${<}\,\mu^+$-completeness of $I_{|t|}$, one of those cases has to occur for $I_{|t|}^+$ many $x$. If it is the first one, we directly have
		$$\{x\in\osucc_p(t)\;|\;\osucc_p(t^{\frown}x)\in I_{|t|+1}\}\in I_{|t|}^+$$
		contradicting the fact that $t\in q$, or more precisely, the second requirement for $n=1$. In the second case, we have that the set
		$$\{r\in\osucc_p^n(t^{\frown}x)\;|\;\osucc_p(t^{\frown}x^{\frown}r)\in I_{|t|+1+n}\}$$
		is in $(I_{|t|+1}\otimes\dots\otimes I_{|t|+n})^+$. Let $A:=\{r\in\osucc_p^{n+1}(t)\;|\;\osucc_p(t^{\frown}r)\in I_{|t|+1+n}\}$. Then the above equation states that
		$$\{x\in [\kappa_{|t|}]^{<\mu^+}\;|\;\{r\in\prod_{i\in[|t|+1,|t|+n]}[\kappa_i]^{<\mu^+}\;|\;(x,r)\in A\}\in(I_{|t|}\otimes\dots\otimes I_{|t|+n})^+\}\in I_{|t|}^+.$$
		In other words, $A\in(I_{|t|}\otimes(I_{|t|+1}\otimes\dots\otimes I_{|t|+n}))^+=(I_{|t|}\otimes\dots\otimes I_{|t|+n})^+$, where the equality follows from Lemma \ref{FubiniAssociative}. But this clearly contradicts the fact that $t\in q$, or more precisely, the second requirement for $n+1$.
		
		In summary, $q$ is a condition in $\dL$ and a $\leq_0$-lower bound of $X$. This finishes the proof.
	\end{proof}
	
	It is easy to see that, after forcing with $\dL(\mu,(\kappa_n,I_n,B_n)_{n\in\omega})$, every $\kappa_n$ has cofinality $\omega$ (but we will not use this): If $b\in\prod_n[\kappa_n]^{<\mu}$ is the generic branch and $k\in\omega$, the function $n\mapsto\sup(b(n)\cap\kappa_k)$ is a cofinal function from $\omega$ into $\kappa_k$. However, $\kappa^*$ does not have such a small cofinality in the extension:
	
	\begin{mylem}\label{KappaStarCof}
		Assume that $\tau$ is an $\dL(\mu,(\kappa_n,I_n,B_n)_{n\in\omega})$-name for an ordinal and $p\in\dL(\mu,(\kappa_n,I_n,B_n)_{n\in\omega})$ forces $\tau<\check{\kappa}^*$. Then there is $\alpha<\kappa^*$ and $q\leq_0p$ which forces $\tau<\check{\alpha}$. Consequently, after forcing with $\dL(\mu,(\kappa_n,I_n,B_n)_{n\in\omega})$, $\cf(\kappa^*)\geq\mu$.
	\end{mylem}
	
	\begin{proof}
		Let $D$ be the open dense set of conditions which decide $\tau$. By Lemma \ref{StrongPrikry} there is $q\leq_0p$ and $n\in\omega$ such that whenever $s\in q$, $|s|=n$, $q\uhr s\in D$, i.e. $q\uhr s\Vdash\tau=\check{\alpha}_s$ for some $\alpha_s$. Since $\{s\in q\;|\;|s|=n\}$ has size $\leq\kappa_n$, there is $\alpha>\alpha_s$ for every $s\in q$ with $|s|=n$. It follows that $q\Vdash\tau<\check{\alpha}$.
		
		Now assume that there is $\gamma<\mu$, $\dot{F}$ and $p\in\dL(\mu,(\kappa_n,I_n,B_n)_{n\in\omega})$ which forces $\dot{F}\colon\check{\gamma}\to\check{\kappa}^*$. Using the fact that the direct extension ordering is in particular ${<}\,\mu$-closed, we can find a condition $q\leq_0p$ and ordinals $(\alpha_{\delta})_{\delta<\gamma}$ such that for every $\delta<\gamma$, $q\Vdash\dot{F}(\check{\delta})<\check{\alpha}_{\delta}$. However, $\sup_{\delta<\gamma}\alpha_{\delta}<\kappa^*$ and so $q$ forces that $\dot{F}$ is bounded in $\check{\kappa}^*$, a contradiction.
	\end{proof}
	
	The following is our crucial strengthening of the main crux of our earlier work (see \cite[Theorem 4.3]{JakobLevineFailureApproachability}). It is a variation of the more well-known \emph{${<}\,\delta$-approximation property} -- that was introduced implicitly by Mitchell in \cite{MitchellTreeProp} and explicitly by Hamkins in \cite{HamkinsExtApprox} -- which holds for a forcing $\dP$ if $\dP$ does not add any new function whose restrictions to all ${<}\,\delta$-sized ground-model sets are in the ground model themselves. On one hand, the assumptions of the theorem are stronger, since we require the function to be unbounded. On the other hand, its conclusion is also stronger, since it states that any such function does not even have an initial segment which is covered by a small ground-model set. This directly implies that $\kappa^*$ does not become approachable. Crucially, it does so without relying on the assumption that the poset does not add any new ``short functions into small sets''. This assumption notably does not hold for our poset, since it can be seen straightforwardly that every regular cardinal between $\mu$ and $\sup_n\kappa_n$ is singularized to have countable cofinality.
	
	\begin{mysen}\label{BetterApproxProp}
		Let $\dot{\dQ}$ be an $\dL(\mu,(\kappa_n,I_n,B_n)_{n\in\omega})$-name for a ${<}\,\mu$-closed forcing. Whenever $\dot{F}\colon\check{\mu}\to\check{\kappa}^*$ is forced to be unbounded by some $(p,\dot{q})\in\dL(\mu,(\kappa_n,I_n,B_n)_{n\in\omega})*\dot{\dQ}$, there is $(p',\dot{q}')\leq_0(p,\dot{q})$ and $i<\mu$ such that there is no $(p'',\dot{q}'')\leq(p',\dot{q}')$ and $x\in V$ with $|x|<\kappa$ such that $(p'',\dot{q}'')\Vdash\im(\dot{F}\uhr\check{i})\subseteq\check{x}$.
	\end{mysen}
	
	\begin{mybem}
		Let us note that, in \cite[Theorem 3]{CoxKruegerNamba}, Cox and Krueger also show that certain variants of Namba forcing possess a weak form of the approximation property. However, their proof is based on the closure of the $\leq_0$-ordering of the Namba forcing. In particular, for our poset, their proof would only show the conclusion for functions on $\delta<\mu$, whereas we crucially need it to apply to functions from $\mu$ into $\kappa$, since no cofinal functions from any $\delta<\mu$ into $\kappa$ are added.
	\end{mybem}
	
	Let us first prove that this Theorem implies that $\kappa^*$ does not become approachable after forcing with $\dL(\mu,(\kappa_n,I_n,B_n)_{n\in\omega})*\dot{\Coll}(\check{\mu},\check{\kappa}^*)$:
	
	\begin{mysen}\label{NoNewApproach}
		Let $d\colon[\kappa^*]^2\to\omega$ be a normal coloring. After forcing with the iteration $\dL(\mu,(\kappa_n,I_n,B_n)_{n\in\omega})*\dot{\Coll}(\check{\mu},\check{\kappa}^*)$, $\kappa^*$ does not become $d$-approachable.
	\end{mysen}
	
	\begin{proof}
		By Lemma \ref{KappaStarCof},  $\dL(\mu,(\kappa_n,I_n,B_n)_{n\in\omega})*\dot{\Coll}(\check{\mu},\check{\kappa}^*)$ forces that $\cf(\kappa^*)=\mu$. Assume toward a contradiction that there is an $\dL(\mu,(\kappa_n,I_n,B_n)_{n\in\omega})*\dot{\Coll}(\check{\mu},\check{\kappa}^*)$-name $\dot{A}$ and a condition $p$ which forces $\dot{A}$ to be a cofinal subset of $\check{\kappa}^*$ on which $\check{d}$ is bounded with value $\check{n}$. Let $\dot{F}$ be an $\dL(\mu,(\kappa_n,I_n,B_n))*\dot{\Coll}(\check{\mu},\check{\kappa}^*)$-name for its increasing enumeration.
		
		By Theorem \ref{BetterApproxProp}, there is $i<\mu$ and $q\leq_0p$ such that $\dot{F}\uhr i$ cannot be forced to be covered by a set of size ${<}\,\kappa$ by any condition below $q$. Let $r\leq q$ decide $\dot{F}(\check{i})$. Then $r$ forces that the image of $\dot{F}\uhr\check{i}$ is covered by $\{\alpha<F(i)\;|\;d(\alpha,F(i))\leq n\}$, since $\dot{F}$ is forced to be increasing and $\check{d}$ is forced to be bounded by $\check{n}$ on the image of $\dot{F}$. But by the normality of $d$, that set has size ${<}\,\kappa$, a contradiction.
	\end{proof}
	
	Now let us prove Theorem \ref{BetterApproxProp}. The proof is similar to the main crux of our earlier work (see \cite{JakobLevineFailureApproachability}) and is roughly based on an argument of Levine from \cite{LevineClassicalNamba}.
	
	\begin{proof}[Proof of Theorem \ref{BetterApproxProp}]
		For simplicity, let $\dL:=\dL(\mu,(\kappa_n,I_n,B_n)_{n\in\omega})$ and let $\dot{F}$ be an $\dL*\dot{\dQ}$-name for a cofinal function from $\check{\mu}$ into $\check{\kappa}^*$. Assume without loss of generality that this is forced by $1_{\dL*\dot{\dQ}}$.
		
		For $i\in\mu$ and $(p,\dot{q})\in\dL*\dot{\dQ}$, we let $\phi(i,p,\dot{q})$ state that there exists a sequence $(A_x)_{x\in\osucc_p(\stem(p))}$ such that for all $x_0,x_1\in\osucc_p(\stem(p))$, $A_{x_0}\cap A_{x_1}=\emptyset$ and
		$$p\uhr(\stem(p)^{\frown}x_0)\Vdash\dot{F}(\check{i})\in\check{A}_{x_0}.$$
		
		\setcounter{myclaim}{0}
		
		\begin{myclaim}
			Assume $(p',\dot{q}')\in\dL*\dot{\dQ}$. Then there is $i\in\mu$ and $(p,\dot{q})\leq_0(p',\dot{q}')$ such that $\phi(i,p,\dot{q})$ holds.
		\end{myclaim}
		
		\begin{proof}
			Let $W:=\osucc_{p'}(\stem(p'))$. By induction on $x\in W$ (along some well-order on $W$) we will define conditions $(p_x,\dot{q}_x)\leq_0(p'\uhr\stem(p')^{\frown}x)$, ordinals $i_x\in(j,\mu)$ and sets $A_x$ of size ${<}\,\kappa$ such that
			$$(p_x,\dot{q}_x)\Vdash\dot{F}(\check{i}_x)\in\check{A}_x.$$
			Let $x:=\min(W)$. Let $i_x\in\mu$ be arbitrary. It follows that there exists $(p_x,\dot{q}_x)\leq_0(p'\uhr(\stem(p)^{\frown}x),\dot{q}')$ and $n_x\in\omega$ such that whenever $s\in p_x$, $|s|=n_x$, $(p_x\uhr s,\dot{q}_x)$ decides $\dot{F}(\check{i}_x)$. Let
			$$A_x:=\{\alpha\;|\;\exists s\in p_x(|s|=n_x\wedge(p_x\uhr s,\dot{q}_x)\Vdash\dot{F}(\check{i}_x)=\check{\alpha})\}.$$
			Clearly, $|A_x|<\kappa$. Suppose the objects have been defined for $y\in x$ and $x$ is in $W$. Let $B:=\bigcup_{y\in W,y<x}A_y$. Then $|B|<\kappa^*$, since $\kappa^*$ is regular. In particular,
			$$(p\uhr(\stem(p)^{\frown}x),\dot{q}')\Vdash\im(\dot{F})\not\subseteq\check{B}$$
			since $\dot{F}$ is forced to be unbounded in $\check{\kappa}^+$ and $B$ is bounded in $\kappa^+$. Therefore, $(p\uhr(\stem(p)^{\frown}x),\dot{q}')\Vdash\exists i\in\check{\mu}(\dot{F}(i)\notin\check{B})$. So we can take $(p_x',\dot{q}_x')\leq_0(p\uhr(\stem(p)^{\frown}x),\dot{q}')$ deciding $i=\check{i}_x$ for some $i_x\in\mu$ (since $\dL*\dot{\dQ}$ has direct $\mu$-decidability by Lemma \ref{LDirectDecide} and Lemma \ref{PrikryIterEastonMu}). So we have
			$$(p_x',\dot{q}_x')\Vdash\dot{F}(\check{i}_x)\notin\check{B}.$$
			Now proceed as in the base case and find $(p_x,\dot{q}_x)\leq_0(p_x',\dot{q}_x')$ and $n_x$ such that whenever $s\in p_x$, $|s|=n_x$, $(p_x\uhr s,\dot{q}_x)$ decides $\dot{F}(\check{i}_x)$ and let $A_x$ be the set of possible decisions (note as before $|A_x|<\kappa$). Since $(p_x',\dot{q}_x')\Vdash\dot{F}(\check{i}_x)\notin\check{B}$, $A_x\cap B=\emptyset$ and so $A_x\cap A_y=\emptyset$ whenever $y\in W$, $y<x$.
			
			Lastly, we can choose $W'\subseteq W$ which is $I_{|\stem(p)|}$-positive such that for some $i\in\mu$ and every $x\in W'$, $i_x=i$. Let $p:=\bigcup_{x\in W'}p_x$. Then clearly $p\in\dL$ and $p\leq_0p'$. Let $\dot{q}$ be such that $p\uhr x\Vdash\dot{q}=\dot{q}_x$ for every $x\in W'$ (this is possible as $\{p\uhr (\stem(p)^{\frown}x)\;|\;x\in W'\}$ is an antichain in $\dL$). Clearly, $(p,\dot{q})$ and $i$ are as required.
		\end{proof}
		
		We now construct a fusion sequence $(p_n)_{n\in\omega}$ and a sequence $(\dot{q}_n)_{n\in\omega}$ such that for all $n\in\omega$, $p_{n+1}\Vdash\dot{q}_{n+1}\dot{\leq}\dot{q}_n$ and for all $s\in p_n$ with $|s|\leq n$ there is some $i_s\in\mu$ such that $\phi(i_s,p\uhr s,\dot{q}_n)$ holds.
		
		Let $(p_0,\dot{q}_0)\leq_01_{\dP*\dot{\dQ}}$ be such that $\phi(i_{\emptyset},p_0,\dot{q}_0)$ holds for some $i_{\emptyset}\in\mu$. Assume $(p_n,\dot{q}_n)$ has been defined. For $s\in p_n$, $|s|=n+1$, we can find $(p_s,\dot{q}_s)\leq_0(p\uhr s,\dot{q}_n)$ and $i_s$ such hat $\phi(p_s,\dot{q}_s,i_s)$ holds. Let $p_{n+1}:=\bigcup_{s\in p_n,|s|=n+1}p_s$ and let $\dot{q}_{n+1}$ be the mixture of $\dot{q}_s$ for $s\in p_{n+1}$, $|s|=n+1$, i.e. for every $s\in p_{n+1}$, $|s|=n+1$, $p_{n+1}\uhr s\Vdash\dot{q}_{n+1}=\dot{q}_s$. Clearly, $p_{n+1}\leq_np_n$, $p_{n+1}\Vdash\dot{q}_{n+1}\dot{\leq}\dot{q}_n$ and they are as required: Whenever $s\in p_{n+1},|s|=n+1$, $p_{n+1}\uhr s=p_s$, $p_{n+1}\uhr s\Vdash\dot{q}_{n+1}=\dot{q}_s$ and they were specifically chosen so that $\phi(p_s,\dot{q}_s,i_s)$ holds.
		
		Now let $p$ be the fusion limit of $(p_n)_{n\in\omega}$. Then $p$ forces the sequence $(\dot{q}_n)_{n\in\omega}$ to be descending in $\dot{\leq}$, so by the maximum principle and the countable closure of $\dot{\dQ}$, we can take $\dot{q}'$ which is forced to be a lower bound.
		
		Let $G\colon[p]\to\mu$ be the function mapping a branch $b\in[p]$ to $\sup_n(i_{b\uhr n})$. For every $\alpha<\mu$, the set $\{b\in[p]\;|\;G(b)<\alpha\}$ is closed in $[p]$ with regards to the tree topology, so by the Rubin-Shelah Theorem (see e.g. \cite[Chapter XI, Lemma 3.5 (1)]{ShelahProperImproper}) and the fact that every $I_n$ is ${<}\,\mu^+$-complete, we can find $p'\leq_0p$ such that $G$ is bounded in $p_0$, say with value $i$. We claim that $(p',\dot{q}')$ and $i$ work.
		
		Assume toward a contradiction that $(p'',\dot{q}'')\leq(p',\dot{q}')$, $x\in V$ has size ${<}\,\kappa$ and $(p'',\dot{q}'')\Vdash\im(\dot{F}\uhr\check{i})\subseteq\check{x}$.
		
		Let $s\in p''$ be such that $\stem(p'')\subseteq s$ and $\kappa_{|s|}>|x|$. We know that $\phi(p''\uhr s,\dot{q}'',i_{|s|})$ holds so there is a sequence $(A_x)_{x\in\osucc_{p''}(s)}$ of disjoint sets such that whenever $x\in\osucc_{p''}(s)$, $(p''\uhr(s^{\frown}x),\dot{q}'')\Vdash\dot{F}(\check{i}_s)\in\check{A}_x$. Since $\stem(p'')\subseteq s$, $\osucc_{p''}(s)$ is $I_{|s|}$-positive and so has size $\kappa_{|s|}$. Thus there has to be $x\in\osucc_{p''}(s)$ such that $A_x\cap y=\emptyset$. But
		$$(p''\uhr(s^{\frown}x),\dot{q}'')\Vdash\dot{F}(\check{i}_{|s|})\in\check{A}_x\cap\check{y}$$
		since $i_{|s|}<i$, a contradiction.
	\end{proof}
	
	\section{Non-approachable Points of a Single Cofinality}
	
	In \cite{JakobLevineFailureApproachability}, Levine and the author introduced a poset which, given a supercompact cardinal $\kappa$ and any $n\geq 1$, collapses $\kappa$ to become $\aleph_{n+1}$ while forcing that there exist stationarily many $\gamma\in\aleph_{\omega+1}\cap\cof(\aleph_n)$ which are not approachable. In this section, we introduce a slight generalization of that poset and more explicitly state and prove its properties. This is the poset which will later be iterated to give us our main theorem.
	
	\begin{mydef}\label{ADef}
		Let $\kappa$ be a supercompact cardinal and $l\colon\kappa\to V_{\kappa}$ a Laver function. Let $\mu<\kappa$ be regular. We define the poset $(\dA(\mu,\kappa),\leq,\leq_0)$ as the limit of an Easton-support Magidor iteration $(\dP_{\alpha},\dot{\dQ}_{\alpha})_{\alpha<\kappa}$ with the following iterands. For $\alpha\leq\mu$, we let $\dot{\dQ}_{\alpha}$ be trivial.
		\begin{itemize}
			\item[Case 1:] $\alpha$ is inaccessible and $|\dP_{\beta}|<\alpha$ for all $\beta<\alpha$. We let $\dP_{\alpha}$ be the direct limit of $(\dP_{\beta},\dot{\dQ}_{\beta})_{\beta<\alpha}$.
			\begin{itemize}
				\item[Case 1.1:] Assume $l(\alpha)=(\dot{\dQ},\delta^*,(\delta_n)_{n\in\omega})$, where $\delta^*>\alpha$ is the successor of a cardinal of countable cofinality, $(\delta_n)_{n\in\omega}$ is an increasing sequence of cardinals $\geq\alpha$ which converges to the predecessor of $\delta^*$ and $(\dot{\dQ},\dot{\leq})$ is a $\dP_{\alpha}$-name for a strongly ${<}\,\alpha$-directed closed poset which forces $\LIP(\check{\mu},\check{\mu}^+,\check{\delta}_n)$ for all $n\in\omega$, witnessed by $(\dot{I}_n,\dot{B}_n)$. In this case, we let $\dot{\dQ}_{\alpha}:=(\dot{\dQ},\dot{\leq},\dot{\leq})$, we let $\dot{\dQ}_{\alpha+1}$ be a $\dP_{\alpha}*(\dot{\dQ},\dot{\leq},\dot{\leq})$-name for $(\dot{\dL}(\check{\mu},(\check{\delta}_n,\dot{I}_n,\dot{B}_n)_{n\in\omega}),\dot{\leq},\dot{\leq}_0)$ and we let $\dot{\dQ}_{\alpha+2}$ be a $\dP_{\alpha}*(\dot{\dQ},\dot{\leq},\dot{\leq})*(\dot{\dL}(\check{\mu},(\check{\delta}_n,\dot{I}_n,\dot{B}_n)_{n\in\omega}),\dot{\leq},\dot{\leq}_0)$-name for $(\dot{\Coll}(\check{\mu},\check{\delta}^*),\dot{\leq},\dot{\leq})$.
				\item[Case 1.2:] Assume $l(\alpha)=\dot{\dQ}$, where $\dot{\dQ}$ is a $\dP_{\alpha}$-name for a strongly ${<}\,\alpha$-directed closed poset. In this case, we let $\dot{\dQ}_{\alpha}:=l(\alpha)$.
			\end{itemize}
			\item[Case 2:] Assume $\alpha$ is a limit ordinal which is not inaccessible. In this case, let $\dP_{\alpha}$ be the inverse limit of $(\dP_{\beta},\dot{\dQ}_{\beta})_{\beta<\alpha}$ and let $\dot{\dQ}_{\alpha}$ be $\dP_{\alpha}$-name for $(\dot{\Coll}(\check{\mu},|\dP_{\alpha}|),\dot{\leq},\dot{\leq})$.
		\end{itemize}
	\end{mydef}
	
	The specific choice of defining $\dot{\dQ}_{\alpha}$, $\dot{\dQ}_{\alpha+1}$ and $\dot{\dQ}_{\alpha+2}$ separately instead of merging them into a single iterand is necessary to ensure that the termspace forcing we are defining later has the required properties.
	
	We will now analyze the poset from the preceding definition. As in the proof of \cite[Theorem 5.3]{JakobLevineFailureApproachability} we can see the following:
	
	\begin{mylem}\label{AProperties}
		Let $\mu<\kappa$ be cardinals such that $\mu$ is regular and $\kappa$ is supercompact.
		\begin{enumerate}
			\item $(\dA(\mu,\kappa),\leq)$ is $\kappa$-cc.
			\item $(\dA(\mu,\kappa),\leq,\leq_0)$ is a Prikry-type poset and has direct $\mu$-decidability.
			\item $(\dA(\mu,\kappa),\leq_0)$ is strongly ${<}\,\mu$-directed closed.
			\item After forcing with $(\dA(\mu,\kappa),\leq)$, all cardinals $\leq\mu$ and $\geq\kappa$ are preserved, while every cardinal in the interval $(\mu,\kappa)$ is collapsed to have size $\mu$.
		\end{enumerate}
	\end{mylem}
	
	\begin{proof}
		(1) follows using standard methods for iterations with Easton support.
		
		(2) follows from Lemma \ref{PrikryIterEastonMu}: Clearly, any poset with $\leq=\leq_0$ has direct $\mu$-decidability. The only iterands we use where this is not the case are the instances of $\dL(\check{\mu},(\check{\delta_n},\dot{I}_n,\dot{B}_n),\dot{\leq},\dot{\leq}_0)$. However, in that case the poset $(\dL((\check{\delta}_n,\dot{I}_n,\dot{B}_n)),\dot{\leq}_0)$ has direct $\mu$-decidability by Lemma \ref{LDirectDecide}.
		
		(3) follows from Lemma \ref{DirectClosureEaston} together with Lemma \ref{LClosure}.
		
		(4) now follows from the previous points together with the fact that for any $\alpha<\kappa$ there is $\beta<\kappa$ with $\beta\geq\alpha$ such that our iteration involves the poset $\Coll(\mu,\beta)$.
	\end{proof}
	
	We now prove the following technical refinement of \cite[Theorem 5.1]{JakobLevineFailureApproachability}:
	
	\begin{mysen}\label{MainThmOneCof}
		Let $\kappa$ be a supercompact cardinal and let $\mu<\kappa$ be regular. Let $(\delta_n)_{n\in\omega}$ be an increasing sequence of regular cardinals with $\delta_0\geq\kappa$ and define $\delta^*:=(\sup_n\delta_n)^*$. Fix a normal, transitive coloring $d\colon[\delta^*]^2\to\omega$. Assume that $\dot{\dP}$ is an $\dA(\mu,\kappa)$-name for a strongly ${<}\,\kappa$-directed closed poset which preserves $\delta^*$, every $\delta_n$ and forces $\LIP(\check{\mu},\check{\mu}^+,\check{\delta}_n)$ for all $n\in\omega$. After forcing with $\dA(\mu,\kappa)$ (obtaining $G$), $\mu$ remains a cardinal, $\kappa=\mu^+$ and for any large enough $\Theta$ there are stationarily many $M\in[H^{V[G]}(\Theta)]^{<\kappa}$ such that the following holds:
		\begin{enumerate}
			\item $\dA(\mu,\kappa),\dot{\dP},d\in M$.
			\item $\cf(\sup(M\cap\delta^*))=\mu$ and $\sup(M\cap\delta^*)$ is not $d$-approachable.
			\item There exists a condition $p\in\dot{\dP}^G$ such that whenever $D\in M$ is open dense in $\dot{\dP}^G$, there is $q\in D\cap M$ such that $p\leq q$.
		\end{enumerate}
	\end{mysen}
	
	\begin{mybem}
		In the context of generalized properness, the condition in (3) is referred to as a \emph{totally $(M,\dP)$-generic condition}.
	\end{mybem}
	
	\begin{proof}
		
		\setcounter{myclaim}{0}
		
		Fix parameters as in the statement of the theorem. For simplicity, let $\dA:=\dA(\mu,\kappa)$. Let $G$ be $\dA$-generic. In $V[G]$, $\mu$ remains a cardinal since no bounded subsets of $\mu$ are added by Lemma \ref{AProperties} (3). By Lemma \ref{AProperties} (4), $\kappa=\mu^+$ in $V[G]$.
		
		Let $\Theta$ be sufficiently large such that $H(\Theta)$ contains all relevant parameters and let $F\colon[H^{V[G]}(\Theta)]^{<\omega}\to[H^{V[G]}(\Theta)]^{<\kappa}$. We want to find some $M\in[H^{V[G]}(\Theta)]^{<\kappa}$ which is closed under $F$ and satisfies the additional statements. To this end, let $\dot{F}$ be a name for $F$ and let $\Theta'$ be such that $\dot{F}\in H^V(\Theta')$.
		
		By Lemma \ref{SmallLaver}, we can find a $\kappa$-Magidor model $N\prec H^V(\Theta')$ containing $\dA$, $\dot{\dP}$, $d$ and $\dot{F}$ such that, with $\pi_N\colon N\to K$ the Mostowski-collapse of $N$, $l(N\cap\kappa)=\pi_N((\dot{\dP},(\delta_n)_{n\in\omega},\delta^*)$.
		
		In $V[G]$, $N[G]\prec H^V(\Theta')[G]=H^{V[G]}(\Theta')$ and so $M:=N[G]\cap H^{V[G]}(\Theta)$ is closed under $F$. By standard arguments (see e.g. \cite[Lemma 5.3.6]{JakobPhD}), the $\kappa$-cc of $\dA$ implies the following:
		
		\begin{myclaim}
			The Mostowski-collapse of $N[G]$ is a map from $N[G]$ to $K[\pi[G\cap N]]$. Moreover, whenever $\tau\in N$:
			\begin{enumerate}
				\item If $\tau^G\in V$, then $\tau^G\in N$. In other words, $N[G]\cap V=N$.
				\item $\pi_{N[G]}(\tau^G)=(\pi_N(\tau))^{\pi[G\cap N]}$. If $\tau^G\in V$, then $\pi_{N[G]}(\tau^G)=\pi_N(\tau^G)$.
			\end{enumerate}
		\end{myclaim}
		
		We want to verify that $M$ witnesses the conclusion of the theorem.
		
		Let $\nu:=M\cap\kappa$ (which is an inaccessible cardinal), $\gamma:=\sup(M\cap\delta^*)$ and $\delta_-^*:=\otp(M\cap\delta^*)$. Since $N[G]\cap V=N$, $\nu=N\cap\kappa$, $\gamma=\sup(N\cap\delta^*)$ and $\delta_-^*=\otp(N\cap\delta^*)$. Then, by Lemma \ref{InducedColoring}, $\delta_-^*=\cf(\gamma)$ is the successor of a singular cardinal of countable cofinality and $d$ induces a normal coloring $e$ on $\delta_-^*$ such that in any outer model, $\gamma$ is $d$-approachable if and only if $\delta_-^*$ is $e$-approachable. For $n\in\omega$, let $\delta_n^-:=\otp(M\cap\delta_n)=\otp(N\cap\delta_n)$. Then each $\delta_n^-$ is a regular cardinal and $\delta_-^*=(\sup_n\delta_n)^+$.
		
		We fix the following notation: Given $\xi<\kappa$, let $\dA(\xi)$ consist of all $f\in\dA(\mu,\kappa)$ such that $\dom(f)\subseteq\xi$, ordered as a suborder of $\dA(\mu,\kappa)$. It is clear that whenever $\xi<\kappa$, the map $p\mapsto p\uhr\xi$ is a projection from $\dA$ to $\dA(\xi)$. So, for any $\xi<\kappa$, we can let $G(\xi)$ be the $\dA(\xi)$-generic filter induced by $G$. It is clear that $G(\xi)=G\cap\dA(\xi)$ whenever $\xi<\kappa$ and that $\pi[G\cap N]=G(\nu)$.
		
		We now prove that $M$ satisfies (2). Clearly, $\dA(\nu)$ is $\nu$-cc since $\nu$ is inaccessible. Moreover, $l(\nu)=\pi_N((\dot{\dP},(\delta_n)_{n\in\omega}),\delta^*)=(\pi_N(\dot{\dP}),(\delta_n^-)_{n\in\omega},\delta^*_-)$.
		
		\begin{myclaim}
			The following holds in $V$:
			\begin{enumerate}
				\item $\delta^*_->\nu$ is the successor of a singular cardinal of countable cofinality;
				\item $(\delta_n^-)_{n\in\omega}$ converges to the predecessor of $\delta^*_-$;
				\item $\pi_N(\dot{\dP})$ is an $\dA(\nu)$-name for a strongly ${<}\,\nu$-directed closed poset which preserves $\delta^*_-$, every $\delta_n^-$ and forces $\LIP(\check{\mu},\check{\mu}^+,\check{\delta}_n^-)$ for every $n\in\omega$.
			\end{enumerate}
		\end{myclaim}
		
		\begin{proof}
			We prove (1), (2) and (3) simultaneously. Since $\Theta$ is sufficiently large, it holds in $H^V(\Theta)$ that
			\begin{enumerate}
				\item $\delta^*>\kappa$ is the successor of a singular cardinal of countable cofinality;
				\item $(\delta_n^-)_{n\in\omega}$ converges to the predecessor of $\delta^*$;
				\item $\dot{\dP}$ is an $\dA$-name for a strongly ${<}\,\nu$-directed closed poset which preserves $\delta^*$, every $\delta_n$ and forces $\LIP(\check{\mu},\check{\mu}^+,\delta_n)$ for every $n\in\omega$.
			\end{enumerate}
			
			Since $\pi_N$ is an isomorphism, the statements (1), (2) and (3) thus hold in $\pi_N(H^V(\Theta))$. By Lemma \ref{MagidorModelElementary}, $\pi_N(H^V(\Theta))=H^V(\pi_N(\Theta))$. Since $\pi_N(\Theta)$ is sufficiently large, this implies that all three statements hold in $V$.
		\end{proof}
		
		Let $(\dot{I}_n,\dot{B}_n)$ be $\dA(\nu)$-names for $\pi_N(\dot{\dP})$-names witnessing item (3). Since $\dA(\nu)$ is just equal to stage $\dP_{\nu}$ of the iteration used to define $\dA$, it follows that
		$$\dA(\nu+3)\cong\dA(\nu)*\pi_N(\dP)*\dL(\check{\mu},(\check{\delta}_n^-,\dot{I}_n,\dot{B}_n)_{n\in\omega})*\dot{\Coll}(\check{\mu},\check{\delta}^*).$$
		
		We now turn to proving the non-$d$-approachability of $\gamma$ in two stages.
		
		\begin{myclaim}\label{AClaim2}
			In $V[G(\nu+3)]$, $\cf(\gamma)=\mu$ and $\gamma$ is not $d$-approachable.
		\end{myclaim}
		
		\begin{proof}
			In $V[G(\nu)]$, $\cf(\gamma)=\delta_-^*$ is still a regular cardinal and $e$ is still a normal, subadditive coloring on $\delta_-^*$, since $\dA(\nu)$ is $\nu$-cc. The same is true in $V[G(\nu+1)]$, since $\pi_N(\dot{\dP})$ is forced to preserve $\delta_-^*$ and every $\delta_n^-$. By Lemma \ref{KappaStarCof} and Theorem \ref{CollapseNoApp}, in $V[G(\nu+3)]$, $\cf(\delta_-^*)=\mu$ and $\delta_-^*$ is not $e$-approachable. Therefore, $\cf(\gamma)=\mu$ and $\gamma$ is not $d$-approachable in the same model.
		\end{proof}
		
		We now verify that the same holds in the final extension $V[G]$. This extension can be seen as a forcing extension of $V[G(\nu+3)]$ using a Prikry-type poset $(\dA(\kappa\smallsetminus\nu+3),\leq,\leq_0)$ such that $(\dA(\kappa\smallsetminus\nu+3),\leq_0)$ is ${<}\,\mu$-closed and has direct $\mu$-decidability (using the same argument as for Lemma \ref{AProperties} (2)). Since there is a strictly increasing and cofinal function from $\mu$ to $\gamma$ in $V[G(\nu+1)]$, the cofinalities of $\mu$ and $\gamma$ are equal in any outer model. Since $\mu$ is a regular cardinal in $V[G]$, $\cf(\gamma)=\mu$ in $V[G]$.
		
		\begin{myclaim}
			In $V[G]$, $\gamma$ is not $d$-approachable.
		\end{myclaim}
		
		\begin{proof}
			In $V[G(\nu+3)]$, fix a function $\Delta\colon\mu\to\gamma$ which is increasing and cofinal. By Fact \ref{FactRefinement}, $\gamma$ is $d$-approachable in $V[G]$ if and only if there is an unbounded subset of $\im(\Delta)$ on which $d$ is bounded. So assume toward a contradiction that $n\in\omega$ and $\dot{\Gamma}$ is an $(\dA(\kappa\smallsetminus\nu+3),\leq)$-name for an increasing function from $\mu$ into $\im(\Delta)$ such that $d$ is bounded on $\im(\dot{\Gamma})$ with value $n$, forced by some $a\in\dA(\kappa\smallsetminus\nu+3)$.
			
			We now construct an increasing sequence $(\alpha_i)_{i<\mu}$ of elements of $\mu$ and a $\leq_0$-decreasing sequence $(a_i)_{i<\mu}$ (with $a_0\leq_0 a$) of elements of $\dA(\kappa\smallsetminus\nu+3)$ such that for any $i<\mu$, $a_i\Vdash\dot{\Gamma}(\check{i})=\check{\Delta}(\check{\alpha}_i)$. If the sequences have been constructed for $j<i$, first let $a_i'$ be a $\leq_0$-lower bound of $(a_j)_{j<i}$ (or $a_i':=a$ if $i=0$). Then
			$$a_i'\Vdash\exists\alpha<\mu(\dot{\Gamma}(\check{i})=\check{\Delta}(\alpha))$$
			so by the maximum principle and the direct $\mu$-decidability of $(\dA(\kappa\smallsetminus\nu+3),\leq,\leq_0)$ we can find $\alpha_i$ and $a_i\leq_0a_i'$ such that
			$$a_i\Vdash\dot{\Gamma}(\check{i})=\check{\Delta}(\check{\alpha}_i).$$
			Since $\dot{\Gamma}$ is forced to be increasing and $\Delta$ is increasing, $\alpha_i>\alpha_j$ for all $j<i$.
			
			Now we claim that $\{\Delta(\alpha_i)\;|\;i<\mu\}$ is a cofinal subset of $\gamma$ on which $d$ is bounded with value $n$. First observe that, because the sequence $(\alpha_i)_{i<\mu}$ is increasing, the set $\{\alpha_i\;|\;i<\mu\}$ is cofinal in $\mu$ and so, since $\Delta$ is increasing and cofinal, $\{\Delta(\alpha_i)\;|\;i<\mu\}$ is cofinal in $\gamma$. Lastly, whenever $j<i<\mu$, we have
			$$a_i\Vdash\check{d}(\check{\Delta}(\check{\alpha}_j),\check{\Delta}(\check{\alpha}_i))=\check{d}(\dot{\Gamma}(\check{j}),\dot{\Gamma}(\check{i}))\leq\check{n}$$
			so $d(\Delta(\alpha_j),\Delta(\alpha_i))\leq n$. Thus, $\{\Delta(\alpha_i)\;|\;i<\mu\}$ witnesses that $\gamma$ is $d$-approachable in $V[G(\nu+3)]$ which contradicts Claim \ref{AClaim2}.
		\end{proof}
		
		Lastly, we prove that $M$ satisfies (3). By construction, we know that $\dA\cong\dA(\nu)*\pi_N(\dot{\dP})*\dot{\dR}$ for some poset $\dot{\dR}$. And so, in $V[G]$ there is a filter $H\subseteq(\pi_N(\dot{\dP}))^{G(\nu)}=\pi_{N[G]}(\dot{\dP}^G)$ which is generic over $V[G(\nu)]$. We note that $(\pi_N(\dot{\dP})^{G(\nu)})\subseteq K$ since $K$ is transitive. By elementarity, $\pi^{-1}[H]$ is a directed subset of $\dot{\dP}^G$ with size ${<}\,\kappa$ and since $\dot{\dP}^G$ is ${<}\,\kappa$-directed closed, there is a condition $p$ with $p\leq q$ for all $q\in\pi^{-1}[H]$. Now assume that $D\in N[G]$ is open dense in $\dot{\dP}^G$. Then $\pi_{N[G]}(D)$ is open dense in $\pi_{N[G]}(\dot{\dP})^{G(\nu)}$ and in $V[G(\nu)]$, since $\pi_{N[G]}(D)=\pi_{N[G]}(\dot{D}^G)=(\pi_N(\dot{D}))^{G(\nu)}$. Ergo, $H\cap\pi_{N[G]}(D)$ contains some condition $q$. Then in particular $q\in K$ and so $\pi_{N[G]}^{-1}(q)\in M\cap \pi_{N[G]}^{-1}[H]\cap D$, since $\pi_{N[G]}$ is an isomorphism. This is what we wanted to show.
	\end{proof}
	
	To prove our main theorem we are going to iterate instances of the forcing $\dA$. Theorem \ref{MainThmOneCof} will be used to show that, thanks to the existence of master conditions and enough decidability using direct extensions, later iterands do not destroy the stationarity of the already added sets of non-approachable points of smaller cofinalities. To show that master conditions as required exist, we want to find a forcing which projects onto $\dA(\mu,\kappa)$ and forces $\LIP(\mu,\mu^+,\delta)$ for all $\delta\geq\kappa$. Note that in general working with the direct extension ordering on a Magidor iteration of Prikry-type forcings is quite difficult since that ordering neither behaves like an iteration nor like a product. However, by Lemma \ref{TermspaceMagidor}, we can take the product of termspace orderings in order to obtain a poset which is more workable.
	
	\begin{mydef}\label{TDef}
		Let $\mu<\kappa$ be cardinals such that $\mu$ is regular and $\kappa$ is supercompact. Let $(\dP_{\alpha},\dot{\dQ}_{\alpha})_{\alpha<\kappa}$ be the iteration used to define $\dA(\mu,\kappa)$ (see Definition \ref{ADef}). We define
		$$\dT(\mu,\kappa):=\prod_{\alpha<\kappa}\dT((\dP_{\alpha},\leq_{\alpha}),(\dot{\dQ}_{\alpha},\dot{\leq}_{\alpha,0}))$$
		where the product is taken with Easton support.
	\end{mydef}
	
	This poset actually works as intended:
	
	\begin{mylem}\label{TProperties}
		Let $\mu<\kappa$ be cardinals such that $\mu$ is regular and $\kappa$ is supercompact.
		\begin{enumerate}
			\item $(\dT(\mu,\kappa),\leq)$ is $\kappa$-cc.
			\item $(\dT(\mu,\kappa),\leq)$ projects onto $(\dA(\mu,\kappa),\leq_0)$.
			\item $(\dT(\mu,\kappa),\leq)$ forces $\check{\kappa}=\check{\mu}^+$.
			\item $(\dT(\mu,\kappa),\leq)$ is strongly ${<}\,\mu$-directed closed.
			\item $(\dT(\mu,\kappa),\leq)$ is a $\LIP(\mu,\kappa)$-forcing.
		\end{enumerate}
	\end{mylem}
	
	\begin{proof}
		(1) follows from standard methods.
		
		(2) follows from Lemma \ref{TermspaceMagidor}.
		
		(3) follows from the fact that $\dT(\mu,\kappa)$ preserves $\kappa$ by the $\kappa$-cc and collapses every $\beta\in(\mu,\kappa)$ to have size $\mu$: For every $\beta<\kappa$ there is some $\alpha\in(\beta,\kappa)$ such that $\dot{\dQ}_{\alpha}$ is a $\dP_{\alpha}$-name for $(\dot{\Coll}(\check{\mu},\check{\alpha}),\dot{\leq},\dot{\leq})$. It follows easily that, in this case, $\dT((\dP_{\alpha},\leq_{\alpha}),(\dot{\dQ}_{\alpha},\dot{\leq}_{\alpha,0}))$ collapses $\alpha$ to have size $\mu$.
		
		(4) follows from Lemma \ref{TermspaceStrongDirect}: By that Lemma (and by Lemma \ref{LClosure}), for every $\alpha<\kappa$, $\dT((\dP_{\alpha},\leq_{\alpha}),(\dot{\dQ}_{\alpha},\dot{\leq}_{\alpha,0}))$ is strongly ${<}\,\mu$-directed closed. This is clearly preserved by the product.
		
		For (5), we only have to verify Definition \ref{LIPForcingDef} (3). Let $C\subseteq\kappa$ consist of all $\alpha\in\kappa$ such that $|\dP_{\beta}|<\alpha$ whenever $\beta<\alpha$. $C$ is clearly club in $\kappa$. Whenever $\nu\in C$ is inaccessible, $\dT(\mu,\kappa)\cap V_{\nu}=\prod_{\alpha<\nu}\dT((\dP_{\alpha},\leq_{\alpha}),(\dot{\dQ}_{\alpha},\dot{\leq}_{\alpha,0}))$ and the projection is just given by $p\mapsto p\uhr\nu$. Let $X\subseteq\dP$ be weakly directed and suppose $p\in\dP\cap V_{\nu}$ is such that $p\leq q\uhr\nu$ for every $q\in X$. Simply let $r\in\dP$ be such that $r\uhr\nu=p$ and $r(i)$ is a lower bound of $\{q(i)\;|\;q\in X\}$ whenever $i\in(\nu,\kappa)$. Then $r$ is easily seen to be as required.
	\end{proof}
	
	\section{The Main Iteration}
	
	In this section we will prove the main theorem. Let us first explain the idea behind the construction: If $\kappa$ is supercompact, the poset $\dA(\mu,\kappa)$ prepares the model as follows: Any further poset which forces $\LIP(\check{\mu},\check{\mu}^+,(\check{\delta}_n)_{n\in\omega})$ for some increasing sequence $(\delta_n)_{n\in\omega}$ of regular cardinals also forces that there are stationarily many points in $(\sup_n\delta_n)^+$ of cofinality $\mu$ which are not approachable. Given this poset, it makes sense to iterate the poset $\dA$ for an increasing sequence of supercompact cardinals. The main problem, of course, is showing that enough instances of $\LIP$ are forced and the stationarity of the set of non-approachable points of cofinality $\mu$ is preserved. The solution to this problem consists of the following two ideas:
	
	\begin{enumerate}
		\item By Theorem \ref{MainThmOneCof}, after forcing with $\dA(\mu,\kappa)$ the stationarity of the set of non-approachable points of cofinality $\mu$ in $\delta^*:=(\sup_n\delta_n)^+$ is witnessed by models $M$ such that there are totally $(M,\dP)$-generic conditions whenever $\dP$ is a ${<}\,\kappa$-directed closed forcing which preserves $\delta^*$, every $\delta_n$ and forces $\LIP(\check{\mu},\check{\mu}^+,\check{\delta}_n)$ for every $n\in\omega$.
		\item Let $\dP$ be the part of the tail forcing which adds subsets of $\delta^*$. Then $\dP$ is a Prikry-type poset which has almost direct $\delta^*$-decidability and a ${<}\,\kappa$-directed closed direct extension ordering. Hence, if we can find a model $M$ with a strong master condition $p$ for $(\dP,\leq_0)$, by almost direct $\delta^*$-decidability $p$ forces $\sup(M\cap\delta^*)$ to be in $\dot{C}$ whenever $\dot{C}\in M$ is a $(\dP,\leq)$-name for a club in $\delta^*$.
	\end{enumerate}
	
	So the main obstacle is showing that $(\dP,\leq_0)$ forces many instances of $\LIP$. However, in general working with the direct extension ordering on a Magidor iteration is complicated, so we work with the product of the termspace forcings instead, which suffices by the projection. To show that this forces $\LIP$, we need that our cardinal $\kappa$ is indestructibly supercompact. By work of Hamkins and Shelah (see \cite{HamkinsShelahSuperdestructibility}), the indestructibility of a supercompact cardinal is destroyed by any nontrivial small forcing. Because of this, we follow every $\dA$-iterand with the poset making the next supercompact cardinal indestructible. The following lemma easily follows from well-known work of Laver (see \cite{LaverIndestruct}) together with our iteration lemma for strong directed closure (see Lemma \ref{StrongDirIter}):
	
	\begin{mylem}
		Let $\mu<\kappa$ be cardinals such that $\mu$ is regular and $\kappa$ is supercompact. There is a poset $\dI(\mu,\kappa)$ such that
		\begin{enumerate}
			\item $\dI(\mu,\kappa)$ is strongly ${<}\,\mu$-directed closed.
			\item After forcing with $\dI(\mu,\kappa)$, $\kappa$ is supercompact and the supercompactness of $\kappa$ is indestructible under strongly ${<}\,\kappa$-directed closed forcing.
		\end{enumerate}
	\end{mylem}
	
	Let us now define the iteration. Let $(\kappa_{\alpha})_{\alpha\in\On}$ be an increasing sequence of cardinals such that $\kappa_0:=\aleph_1$, $\kappa_{\alpha}$ is supercompact whenever $\alpha$ is a successor and $\kappa_{\alpha}=\sup_{\beta<\alpha}\kappa_{\beta}$ whenever $\alpha$ is a limit. For any ordinal $\alpha$, let $\kappa_{\alpha}^*$ be the smallest regular cardinal above $\kappa_{\alpha}$, i.e. $\kappa_{\alpha}^*=\kappa_{\alpha}$ whenever $\alpha$ is a successor or $\alpha$ is a regular cardinal such that $\kappa_{\alpha}=\alpha$ and $\kappa_{\alpha}^*=\kappa_{\alpha}^+$ otherwise.
	
	Now we can construct our iteration as follows:
	
	\begin{mydef}
		Let $\dP$ be the direct limit of the full-support Magidor iteration $((\dP_{\beta},\leq_{\beta},\leq_{\beta,0}),(\dot{\dQ}_{\beta},\dot{\leq}_{\beta},\dot{\leq}_{\beta,0})_{\beta\in\On}$, where for each $\beta$, $\dot{\dQ}_{\beta}$ is a $\dP_{\beta}$-name for the Magidor iteration $(\dot{\dI}(\check{\kappa}_{\beta}^*,\check{\kappa}_{\beta+1}),\dot{\leq},\dot{\leq})*(\dot{\dA}(\check{\kappa}_{\beta},\check{\kappa}_{\beta+1}),\dot{\leq},\dot{\leq}_0)$.
	\end{mydef}
	
	We will verify in a series of lemmas that $\dP$ is as desired. Since $\dP$ is not a set, but a class, ``forcing with $\dP$'' is ill-defined -- consider a forcing adding $\On$-many Cohen reals. However, in our case we will show the following:
	
	\begin{enumerate}
		\item The iteration is well-behaved so that ``forcing with $\dP$'' actually leads to a model of $\ZFC$.
		\item After forcing with $\dP$, for any singular $\delta<\kappa$ of countable cofinality and every regular $\mu<\delta$, there are stationarily many $\gamma\in\delta^+\cap\cof(\mu)$ which are not approachable.
	\end{enumerate}
	
	Of particular importance is the behavior of the tail forcings. For $\alpha<\beta$, let $(\dot{\dP}_{\alpha,\beta},\dot{\leq}_{\alpha,\beta},\dot{\leq}_{(\alpha,\beta),0})$ be a $\dP_{\alpha}$-name for a poset such that $(\dP_{\beta},\leq_{\beta},\leq_{\beta,0})\cong(\dP_{\alpha},\leq_{\alpha},\leq_{\alpha,0})*(\dot{\dP}_{\alpha,\beta},\dot{\leq}_{\alpha,\beta},\dot{\leq}_{(\alpha,\beta),0})$. We have the following:
	
	\begin{mylem}\label{TailClosure}
		For any $\alpha<\beta$, $(\dot{\dP}_{\alpha,\beta},\dot{\leq}_{\alpha,\beta},\dot{\leq}_{(\alpha,\beta),0})$ is forced to be a Prikry-type forcing such that $(\dot{\dP}_{\alpha,\beta},\dot{\leq}_{(\alpha,\beta),0})$ is strongly ${<}\,\check{\kappa}_{\alpha}^*$-directed closed.
	\end{mylem}
	
	\begin{proof}
		This is clear, as $(\dot{\dP}_{\alpha,\beta},\dot{\leq}_{\alpha,\beta},\dot{\leq}_{(\alpha,\beta),0})$ is forced to be a full-support Magidor iteration of Prikry-type forcings starting with $\dot{\dI}(\check{\kappa}_{\alpha}^*,\check{\kappa}_{\alpha+1})*\dot{\dA}(\check{\kappa}_{\alpha}^*,\check{\kappa}_{\alpha+1})$. Therefore, any iterand is forced to be strongly ${<}\,\check{\kappa}_{\alpha}^*$-directed closed, so we can apply Lemma \ref{FullSupportClosure}.
	\end{proof}
	
	In particular, for any $\alpha<\beta$, forcing with $\dot{\dP}_{\alpha,\beta}^G$ (where $G$ is $\dP_{\alpha}$-generic) does not add any new bounded subsets of $\kappa_{\alpha}$ or $\kappa_{\alpha}^+$. We may therefore construct a forcing extension ``by $\dP$'' as follows: Let $G\subseteq\dP$ be a filter which intersects every set-sized dense subclass of $\dP$. For any $\alpha\in\On$, $G\uhr\alpha:=\{p\uhr\alpha\;|\;p\in G\}$ is a generic filter over $\dP_{\alpha}$. Let $H^*(\kappa_{\alpha})$ be equal to $H(\kappa_{\alpha})$ as computed in $V[G_{\alpha}]$. It follows that, for any $\beta>\alpha$, $H^*(\kappa_{\alpha})$ is equal to $H(\kappa_{\alpha})$ as computed in $V[G_{\beta}]$, since $\dot{\dP}_{\alpha,\beta}^{G_{\alpha}}$ does not add any new elements of $H(\kappa_{\alpha})$. Lastly, let
	$$V[G]:=\bigcup_{\alpha\in\On}H^*(\kappa_{\alpha}).$$
	It follows that $V[G]$ is a model of $\ZFC$, since any axiom can be verified locally. For example, the powerset axiom holds in $V[G]$ as follows: Any $x\in V[G]$ is in some $H^*(\kappa_{\alpha})$. Thus, any subset of $x$ is an element of $H^*(\kappa_{\alpha})$ as well. Therefore, the powerset of $x$ in $V[G]$ is equal to the powerset of $x$ in $V[G\uhr\alpha]\subseteq V[G\uhr\alpha+1]$ and thus a member of $H^*(\kappa_{\alpha+1})\subseteq V[G]$.
	
	Moreover, for any $\alpha\in\On$, $H(\kappa_{\alpha})$ as computed in $V[G]$ is equal to $H^*(\kappa_{\alpha})$ and thus equal to $H(\kappa_{\alpha})$ as computed in $V[G\uhr\alpha]$. Therefore, any bounded assertion can be verified by proving that it holds in $V[G\uhr\beta]$ for all sufficiently large $\beta$.
	
	For more information on class-sized forcing extensions, see e.g. the proof of Easton's Theorem in Jech's textbook\cite[Page 235]{JechSetTheory}. Class forcing can be avoided at the cost of increasing the large cardinal assumptions by instead starting with a sequence $(\kappa_{\alpha})_{\alpha<\delta}$, where $\delta$ is inaccessible and then performing the set-sized iteration up to $\delta$, whence the conclusion of the main theorem holds in $V_{\delta}$.
	
	Additionally, we show that the direct extension ordering on the tail forcing can decide ``enough'':
	
	\begin{mylem}\label{TailAlmostDecide}
		Let $\alpha<\gamma$ be ordinals such that $\gamma$ is a limit with countable cofinality. Let $G$ be $\dP_{\alpha}$-generic and $\dP_{\alpha,\gamma}:=\dot{\dP}_{\alpha,\gamma}^G$. Then $(\dP_{\alpha,\gamma},\leq_{\alpha,\gamma},\leq_{(\alpha,\gamma),0})$ has almost direct $\kappa_{\gamma}^+$-decidability.
	\end{mylem}
	
	\begin{proof}
		This is clear by Lemma \ref{DirectDecideCapture}: For any $\beta<\gamma$, $\dP_{\alpha,\beta}$ forces that $(\dot{\dQ}_{\beta},\dot{\leq}_{\beta},\dot{\leq}_{\beta,0})$ has size ${<}\,\check{\kappa}_{\gamma}^+$ and thus the $(\check{\kappa}_{\gamma}^+,\check{\kappa}_{\gamma}^+)$-covering property. By Lemma \ref{DirectDecideCapture}, $(\dP_{\alpha,\gamma},\leq_{\alpha,\gamma},\leq_{(\alpha,\gamma),0})$ has the direct $(\kappa_{\gamma}^+,\kappa_{\gamma}^+)$-covering property as well. In particular, since $\kappa_{\gamma}^+$ is regular, $(\dP_{\alpha,\gamma},\leq_{\alpha,\gamma},\leq_{(\alpha,\gamma),0})$ has almost direct $\kappa_{\gamma}^+$-decidability.
	\end{proof}
	
	We now analyze the cardinal structure of the forcing extension by $\dP$.
	
	\begin{mylem}\label{IterPresCard}
		After forcing with $\dP$, the class of infinite cardinals is given by $\{\aleph_0\}\cup\{\kappa_{\alpha},\kappa_{\alpha}^*\;|\;\alpha\in\On\}$.
	\end{mylem}
	
	\begin{proof}
		First assume that $\alpha\in\On$ is a successor ordinal, $\alpha=\beta+1$. Then $\kappa_{\alpha}>\kappa_{\beta}$ is supercompact. It follows that $|\dP_{\beta}|<\kappa_{\alpha}$ and so $\kappa_{\alpha}$ is supercompact after forcing with $\dP_{\beta}$. After forcing with $\dP_{\alpha}=\dP_{\beta}*\dot{\dI}(\check{\kappa}_{\beta}^*,\check{\kappa}_{\beta+1})*\dot{\dA}(\check{\kappa}_{\beta}^*,\check{\kappa}_{\alpha})$, $\kappa_{\alpha}$ is preserved and every cardinal below $\kappa_{\beta}^*$ and $\kappa_{\alpha}$ is collapsed by Lemma \ref{AProperties} (1) and (4), so $\kappa_{\alpha}$ becomes $(\kappa_{\beta}^*)^+$. Lastly, for every $\gamma>\alpha$, $\dot{\dP}_{\alpha,\gamma}$ is forced to be a Prikry-type forcing such that $(\dot{\dP}_{\alpha,\gamma},\dot{\leq}_{(\alpha,\gamma),0})$ is strongly ${<}\,\check{\kappa}_{\alpha}$-directed closed. In particular, $\dot{\dP}_{\alpha,\gamma}$ is forced to preserve that $\check{\kappa}_{\alpha}$ is a cardinal.
		
		If $\alpha\in\On$ is a limit, then $\kappa_{\alpha}$ is preserved since cofinally many $\nu<\kappa_{\alpha}$ (namely, all $\kappa_{\beta}$, $\beta<\alpha$) are preserved. As for $\kappa_{\alpha}^+$ (in case $\alpha$ is singular or $\kappa_{\alpha}\neq\alpha$): Assume first that $\dP_{\alpha}$ collapses $\kappa_{\alpha}^+$. Since $\kappa_{\alpha}$ is singular (if $\alpha$ is singular, then clearly $\kappa_{\alpha}$ is singular, otherwise $\kappa_{\alpha}$ is singular because $\cf(\kappa_{\alpha})=\alpha<\kappa_{\alpha}$), it follows that there is $\nu<\kappa_{\alpha}$ such that $\cf(\kappa_{\alpha}^+)=\nu$ after forcing with $\dP_{\alpha}$. There is $\beta<\alpha$ such that $\kappa_{\beta}>\nu$ and $\kappa_{\beta}>\cf(\kappa_{\alpha})$.
		
		Let $G_{\beta}$ be $\dP_{\beta}$-generic. We want to show that $\cf(\kappa_{\alpha}^+)=\nu$ in $V[G_{\beta}]$, obtaining a contradiction from the fact that $|\dP_{\beta}|<\kappa_{\alpha}^+$. Let $\dP_{\beta,\alpha}:=\dot{\dP}_{\beta,\alpha}^{G_{\beta}}$. By Lemma \ref{TailAlmostDecide}, $\dP_{\beta,\alpha}$ has almost direct ${<}\,\kappa_{\alpha}^+$-decidability. It follows that every function from $\nu$ to $\kappa_{\alpha}^+$ added by $\dP_{\beta,\alpha}$ is bounded pointwise by a function from $\nu$ to $\kappa_{\alpha}^+$ in $V[G_{\beta}]$ (as in Lemma \ref{KappaStarCof}). In particular $\cf(\kappa_{\alpha}^+)=\nu$ in $V[G_{\beta}]$. As mentioned before, this presents a contradiction, as $|\dP_{\beta}|<\kappa_{\alpha}^+$.
		
		So we have shown that all $\kappa_{\alpha}$ and $\kappa_{\alpha}^*$ remain cardinals. Clearly, whenever $\mu$ does not have this form, $\mu$ is collapsed. So the above class contains precisely all cardinals in the forcing extension.
	\end{proof}
	
	It follows that the singular cardinals of the forcing extension are precisely $\kappa_{\alpha}$ for $\alpha$ a limit and thus that in the extension the successors of singular cardinals of countably cofinality are precisely $\kappa_{\alpha}^+$ where $\cf(\kappa_{\alpha})=\omega$ (this occurs if and only if $\cf(\alpha)=\omega$ since there is always an unbounded and increasing function from $\alpha$ into $\kappa_{\alpha}$).
	
	We want to show that $\dP$ gives us exactly the model we need. The proof idea is as follows: The regular cardinals of the forcing extension are precisely the $\kappa_{\alpha}$ for $\alpha$ a successor and $\kappa_{\alpha}^+$ for $\alpha$ a limit. For any such cardinal, assuming enough instances of $\LIP$, $\dA(\kappa_{\alpha}^*,\kappa_{\alpha+1})$ forces the existence of stationarily many non-approachable points of cofinality $\kappa_{\alpha}^*$ in any successor of a singular cardinal $>\kappa_{\alpha+1}$ of countable cofinality. Then all that is left is to show that the tail of the iteration does not destroy that stationarity (since it cannot make points of cofinality $\kappa_{\alpha}^*$ approachable due to its closure). For this we use that we have used a guessing function and thus added many partially generic filters to many strongly ${<}\,\kappa_{\alpha}$-directed closed forcings, namely, those forcings which force enough instances of $\LIP$. We now show that the direct extension ordering on the tail of the iteration can be projected onto from one of those suitable forcings.
	
	\begin{mydef}
		Let $\alpha<\beta$ be ordinals. Let $G_{\alpha+1/2}=G_{\alpha}*H_{\alpha}$ be $\dP_{\alpha}*\dot{\dI}(\check{\kappa}_{\alpha}^*,\check{\kappa}_{\alpha+1})$-generic and work in $V[G_{\alpha+1/2}]$ (replacing every $\dot{\dQ}_{\gamma}$ for $\gamma\geq\alpha$ by the corresponding $\dot{\dP}_{\alpha,\gamma}^{G_{\alpha+1/2}}$-name). Define
		$$\dT_{\alpha+1/2,\beta}:=\left(\dT(\kappa_{\alpha}^*,\kappa_{\alpha+1})\times\prod_{\gamma\in(\alpha,\beta)}\dT((\dP_{\alpha,\gamma}^{G_{\alpha+1/2}},\dot{\leq}_{\alpha,\gamma}^{G_{\alpha+1/2}}),(\dot{\dQ}_{\gamma},\dot{\leq}_{\gamma,0}))\right)$$
		where $\dT(\kappa_{\alpha}^*,\kappa_{\alpha+1})$ is as in Definition \ref{TDef}.
		Working in $V[G_{\alpha}]$, let $\dot{\dT}_{\alpha+1/2,\beta}$ be an $\dot{\dI}(\check{\kappa}_{\alpha}^*,\check{\kappa}_{\alpha+1})$-name for $\dT_{\alpha+1/2,\beta}$ and let
		$$\dT_{\alpha,\beta}:=\dI(\kappa_{\alpha}^*,\kappa_{\alpha+1})*\dot{\dT}_{\alpha+1/2,\beta}.$$
	\end{mydef}
	
	We use this mixture of a product and an iteration because otherwise the remainder of the termspace forcing would not be directed-closed in the model \emph{after} making the supercompactness of $\kappa_{\alpha+1}$ indestructible under ${<}\,\kappa_{\alpha+1}$-directed closed forcings.
	
	We now show that $\dT_{\alpha,\beta}$ is indeed one of those suitable forcings which were incorporated into the definition to construct $\dA(\mu,\kappa)$:
	
	\begin{mylem}\label{TProp}
		Let $\alpha<\beta$ be ordinals. Let $G_{\alpha}$ be $\dP_{\alpha}$-generic and work in $V[G_{\alpha}]$ (replacing every $\dot{\dQ}_{\gamma}$ for $\gamma\geq\alpha$ by the corresponding $\dot{\dP}_{\alpha,\gamma}^{G_{\alpha}}$-name).
		\begin{enumerate}
			\item There is a projection from $\dT_{\alpha,\beta}$ onto $(\dot{\dP}_{\alpha,\beta}^{G_{\alpha}},\dot{\leq}_{(\alpha,\beta),0}^{G_{\alpha}})$.
			\item $\dT_{\alpha,\beta}$ is strongly ${<}\,\kappa_{\alpha}^*$-directed closed.
			\item For every $\gamma\geq\alpha+1$, $\dT_{\alpha,\beta}$ preserves $\kappa_{\gamma}$ and forces $\LIP(\kappa_{\alpha}^*,(\kappa_{\alpha}^*)^+,\kappa_{\gamma})$.
		\end{enumerate}
	\end{mylem}
	
	\begin{proof}
		For (1), by Lemma \ref{TProperties}, in $V[G_{\alpha}]$, $\dI(\kappa_{\alpha}^*,\kappa_{\alpha+1})$ forces that there is a projection from $\dot{\dT}_{\alpha+1/2,\beta}$ onto $$\dT(\kappa_{\alpha}^*,\kappa_{\alpha+1})\times\prod_{\gamma\in(\alpha,\beta)}\dT((\dP_{\alpha,\gamma}^{G_{\alpha+1/2}},\dot{\leq}_{\alpha,\gamma}^{G_{\alpha+1/2}}),(\dot{\dQ}_{\gamma},\dot{\leq}_{\gamma,0})).$$
		By Lemma \ref{TermspaceMagidor}, it is forced that this forcing projects onto $\dot{\dA}(\check{\kappa}_{\alpha}^*,\check{\kappa}_{\alpha+1})*\dot{\dP}_{\alpha+1,\beta}$.
		
		This clearly implies the existence of a projection from $\dT_{\alpha,\beta}=\dI(\kappa_{\alpha}^*,\kappa_{\alpha+1})*\dot{\dT}_{\alpha+1/2,\beta}$ onto $\dI(\kappa_{\alpha}^*,\kappa_{\alpha+1})*\dot{\dA}(\check{\kappa}_{\alpha}^*,\check{\kappa}_{\alpha+1})*\dot{\dP}_{\alpha+1,\beta}=\dP_{\alpha,\beta}$.
		
		(2) is also clear: For any $\gamma\in[\alpha,\beta)$, $(\dP_{\alpha,\gamma},\leq_{\alpha,\gamma})$ forces that $(\dot{\dQ}_{\gamma},\dot{\leq}_{\gamma,0})$ is strongly ${<}\,\kappa_{\alpha}^*$-directed closed. Thus, the termspace forcing has the same degree of closure by Lemma \ref{TermspaceStrongDirect}. It follows that $\dI(\kappa_{\alpha}^*,\kappa_{\alpha+1})$ forces that $\dot{\dT}_{\alpha+1/2,\beta}$ is strongly ${<}\,\kappa_{\alpha}^*$-directed closed. Since $\dI(\kappa_{\alpha}^*,\kappa_{\alpha+1})$ is strongly ${<}\,\kappa_{\alpha}^*$-directed closed, the iteration (which is equal to $\dT_{\alpha,\beta}$) is strongly ${<}\,\kappa_{\alpha}^*$-directed closed.
		
		We now prove (3). We will show that $\dI(\kappa_{\alpha}^*,\kappa_{\alpha+1})$ forces that $\dot{\dT}_{\alpha+1/2,\beta}$ forces $\LIP(\check{\kappa}_{\alpha}^*,(\check{\kappa}_{\alpha}^*)^+,\check{\kappa}_{\gamma})$ for every $\gamma\geq\alpha+1$, which clearly suffices. So let $H_{\alpha}$ be $\dI(\kappa_{\alpha}^*,\kappa_{\alpha+1})$-generic over $V[G_{\alpha}]$. In $V[G_{\alpha}*H_{\alpha}]$, the supercompactness of $\kappa_{\alpha+1}$ is preserved by any strongly ${<}\,\kappa_{\alpha+1}$-directed closed forcing. In particular, it is preserved by $\dT:=\prod_{\gamma\in(\alpha,\beta)}\dT((\dP_{\alpha,\gamma}^{G_{\alpha}*H_{\alpha}},\dot{\leq}_{\alpha,\gamma}^{G_{\alpha}*H_{\alpha}}),(\dot{\dQ}_{\gamma},\dot{\leq}_{\gamma,0}))$. By Lemma \ref{TProperties} (5),  $\dT(\kappa_{\alpha}^*,\kappa_{\alpha+1})$ is a $\LIP$-forcing, which is preserved by $\dT$. Ergo, $\dT$ followed by $\dT(\kappa_{\alpha}^*,\kappa_{\alpha+1})$ forces $\LIP(\kappa_{\alpha}^*,(\kappa_{\alpha}^*)^+,\kappa_{\gamma})$ for every $\gamma\geq\alpha+1$, since it collapses $\kappa_{\alpha+1}$ to become $(\kappa_{\alpha}^*)^+$. The preservation of each $\kappa_{\gamma}$ follows just as in Lemma \ref{IterPresCard}.
	\end{proof}
		
	Now we can finally prove our main theorem:
	
	\begin{mysen}
		After forcing with $\dP$, $\GCH$ holds and whenever $\delta$ is singular of countable cofinality and $\mu<\delta$ is regular uncountable, there are stationarily many $\gamma<\delta^+$ of cofinality $\mu$ which are not approachable.
	\end{mysen}
	
	\begin{proof}
		By Lemma \ref{IterPresCard}, it suffices to show that whenever $\gamma$ is an ordinal with countable cofinality and $\alpha<\gamma$, then after forcing with $\dP$ there are stationarily many $\beta\in\kappa_{\gamma}^+$ of cofinality $\kappa_{\alpha}^*$ which are not approachable, since the class of regular cardinals consists precisely of the $\kappa_{\alpha}^*$. So fix, in the ground model, a normal coloring $d\colon[\kappa_{\gamma}^+]^2\to\omega$. It suffices to consider only one specific coloring by Fact \ref{ColoringCanonical}.
		
		Since for any $\gamma'>\gamma$, $(\dot{\dP}_{\gamma,\gamma'},\dot{\leq}_{(\gamma,\gamma'),0})$ is forced to be ${<}\,\kappa_{\gamma}^+$-closed (and to have direct $\kappa_{\gamma}^+$-decidability by Lemma \ref{ClosureDecide}), $(\dot{\dP}_{\gamma,\gamma'},\dot{\leq}_{\gamma,\gamma'})$ is forced to preserve stationary subsets of $\check{\kappa}_{\gamma}^+$ (by the usual arguments). Therefore it suffices to show that the statement holds after forcing with $\dP_{\gamma}$. We write 
		$$\dP_{\gamma}\cong\dP_{\alpha}*\dot{\dI}(\check{\kappa}_{\alpha}^*,\check{\kappa}_{\alpha+1})*\dot{\dA}(\check{\kappa}_{\alpha}^*,\check{\kappa}_{\alpha+1})*\dot{\dP}_{\alpha+1,\gamma}.$$
		
		Let $G$ be $\dP_{\alpha}*\dot{\dI}(\check{\kappa}_{\alpha}^*,\check{\kappa}_{\alpha+1})*\dot{\dA}(\check{\kappa}_{\alpha}^*,\check{\kappa}_{\alpha+1})$-generic and work in $V[G]$. Assume that $\dot{C}$ is a $\dP_{\alpha+1,\gamma}:=\dot{\dP}_{\alpha+1,\gamma}^G$-name $\dot{C}$ for a club in $\check{\kappa}_{\gamma}^+$, forced by some condition $p\in\dP_{\alpha+1,\gamma}$. Without loss of generality, assume $p=1_{\dP_{\alpha+1,\gamma}}$ (otherwise, work with $\dP_{\alpha+1,\gamma}\uhr p$ instead). Let $\sigma\colon\dT_{\alpha+1,\gamma}\to(\dP_{\alpha+1,\gamma},\leq_{(\alpha+1,\gamma),0})$ be the projection obtained from Lemma \ref{TProp} (1).
		
		By Theorem \ref{MainThmOneCof} and Lemma \ref{TProp} (2) and (3), let $M\prec H^{V[G]}(\Theta)$ with $|M|<\kappa_{\alpha+1}$ be such that $M$ contains all relevant parameters, $\dot{C}\in M$, $\sup(M\cap\kappa_{\gamma}^+)$ is not $d$-approachable and has cofinality $\kappa_{\alpha}^*$ and there is a condition $p\in\dT_{\alpha+1,\gamma}$ such that whenever $D\in M$ is open dense in $\dT_{\alpha+1,\gamma}$, there is $q\in D\cap M$ with $p\leq q$. We claim that $\sigma(p)\in\dP_{\alpha+1,\gamma}$ forces (over the non-direct ordering $(\dP_{\alpha+1,\gamma},\leq_{\alpha+1,\gamma})$) that $\sup(M\cap\kappa_{\gamma}^+)\in\dot{C}$. Since clearly $\sup(M\cap\kappa_{\gamma}^+)$ does not become $d$-approachable by the distributivity of $\dP_{\alpha+1,\gamma}$, this suffices.
		
		Let $\beta<\sup(M\cap\kappa_{\gamma}^+)$, i.e. $\beta<\beta'$ for some $\beta'\in M\cap\kappa_{\gamma}^+$. Since $\dot{C}$ is forced to be unbounded in $\kappa_{\gamma}^+$, $\Vdash\exists\xi(\xi\in\dot{C}\wedge\beta'<\xi)$. Since $(\dP_{\alpha+1,\gamma},\leq_{\alpha+1,\gamma},\leq_{(\alpha+1,\gamma),0})$ has almost direct $\kappa_{\gamma}^+$-decidability by Lemma \ref{TailAlmostDecide}, the set $D$ of all $q\in\dP_{\alpha+1,\gamma}$ such that for some $\zeta\in\kappa_{\gamma}^+$, $q\Vdash\exists\xi<\check{\zeta}(\xi\in\dot{C}\wedge\beta'<\xi)$, is open dense in $(\dP_{\alpha+1,\gamma},\leq_{(\alpha+1,\gamma),0})$. Hence $\sigma^{-1}[D]$ is open dense in $\dT_{\alpha+1,\gamma}$ and so there is $r\in\sigma^{-1}[D]\cap M$ with $p\leq r$, which implies $\sigma(r)\in D\cap M$, $\sigma(p)\leq_{(\alpha+1,\gamma),0}\sigma(r)$. The corresponding $\zeta$ is in $M$ as well by elementarity.
		
		In summary, for every $\beta<\sup(M\cap\kappa_{\gamma}^+)$ there is $\zeta<\sup(M\cap\kappa_{\gamma}^+)$ with $\beta<\zeta$ such that $\sigma(p)\Vdash\dot{C}\cap(\beta,\zeta)\neq\emptyset$. This clearly implies that $\sigma(p)$ forces that $\sup(M\cap\kappa_{\gamma}^+)\cap\dot{C}$ is unbounded in $\sup(M\cap\kappa_{\gamma}^+)$, which implies that $\sigma(p)$ forces $\sup(M\cap\kappa_{\gamma}^+)\in\dot{C}$.
	\end{proof}
	
	\printbibliography
\end{document}